\documentclass[12pt]{article} 
\usepackage{array,epsfig,fancyhdr,rotating}
\usepackage{tcolorbox}
\usepackage{sectsty, secdot}

\usepackage{amsmath}

\RequirePackage[OT1]{fontenc}
\RequirePackage{amsthm,amsmath}
\RequirePackage{amssymb}

\usepackage[round]{natbib}
\RequirePackage[colorlinks,citecolor=blue,urlcolor=blue]{hyperref}
\usepackage{multirow}
\usepackage[active]{srcltx}
\usepackage{MnSymbol}
\usepackage{graphicx}
\usepackage{verbatim,enumerate}
\usepackage{rotating, lscape}
\usepackage{setspace}
\usepackage{tabularx}
\usepackage{tabu}
\usepackage{booktabs}

\def\S{\mathbb{S}}
\def\R{\mathbb{R}}

\def\T{\R}

\def\bZ{\boldsymbol{Z}}

\def\Unit{\boldsymbol{1}}

\numberwithin{equation}{section}
\theoremstyle{plain}
\newtheorem{thm}{Theorem}[section]
\newtheorem{rem}{Remark}[section]
\newtheorem{lem}{Lemma}[section]

\def\btau{\boldsymbol{\tau}}

\def\bZ{\boldsymbol{Z}}
\def\bz{\boldsymbol{z}}

\def\btheta{\boldsymbol{\theta}}

\def\br{\boldsymbol{r}}

\def\ss{\boldsymbol{s}}

\def\bX{\boldsymbol{\chi}}

\def\bm{\boldsymbol}

\makeatletter
\def\varphiall{\@ifnextchar[{\varphiall@i}{\varphiall@i[]}}
\def\varphiall@i[#1]{\@ifnextchar[{\varphiall@ii{#1}}{\varphiall@ii{#1}[#1]}}
\def\varphiall@ii#1[#2]{\varphi_{\mu,\kappa,\alpha,\beta_{#1},\sigma^2_{#2}}}
\def\hatvarphiall{\@ifnextchar[{\hatvarphiall@i}{\hatvarphiall@i[]}}
\def\hatvarphiall@i[#1]{\@ifnextchar[{\hatvarphiall@ii{#1}}{\hatvarphiall@ii{#1}[#1]}}
\def\hatvarphiall@ii#1[#2]{\widehat{\varphi}_{\mu,\kappa,\alpha,\beta_{#1},\sigma^2_{#2}}}
\makeatother

\def\R{\mathbb{R}}

\def\S{\mathbb{S}}
\def\d{\textrm{d}}

\newcommand{\norm}[1]{\left\Vert#1\right\Vert}

\newmuskip\pFqmuskip

\newcommand*\pFq[6][8]{%
  \begingroup 
  \pFqmuskip=#1mu\relax
  \mathcode`\,=\string"8000
  \begingroup\lccode`\~=`\,
  \lowercase{\endgroup\let~}\pFqcomma
  {}_{#2}F_{#3}{\left[\genfrac..{0pt}{}{#4}{#5};#6\right]}%
  \endgroup
}
\newcommand{\pFqcomma}{\mskip\pFqmuskip}

\begin{document}

\begin{center}

  \vskip15mm
  { \LARGE{\bf Compatibility of Space-Time Kernels with Full, Dynamical, or Compact Support}}

  \vskip15mm
    \large
 Tarik Faouzi,\footnote{
Department of Mathematics and Computer Science, University of Santiago of Chile, Santiago, Chile.
E-mail: tarik.faouzi@usach.cl
}
Reinhard Furrer\footnote{
Department of Mathematics and Institute of Computational Science, University of Zurich, Switzerland.
E-mail: reinhard.furrer@math.uzh.ch
}
and
 Emilio Porcu\footnote{
Department of Mathematics at Khalifa University, Abu Dhabi, UAE.
and RDISC. 
E-mail: georgepolya01@gmail.com}

\end{center}

\vspace{2cm}

\begin{abstract}
We deal with the comparison of space-time covariance kernels having, either, full, spatially dynamical, or space-time compact support. Such a comparison is based on compatibility of these covariance models under fixed domain asymptotics, having a theoretical background that is substantially coming from equivalence or orthogonality of Gaussian measures. In turn, such a theory is intimately related to the tails of the spectral densities associated with the three models. \\
Models with space-time compact support are still elusive. We taper the temporal part of a model with dynamical support, obtaining a space-time compact support. The spectrum related to such a construction is obtained through temporal convolution of the spatially dynamical spectrum with the spectrum associated with the temporal taper. The solution of such a challenge opens the door to the compatibility-based comparison. Our findings show that indeed these three models can be compatible under some suitable parametric restrictions.
As a corollary, we deduce implications in terms of maximum likelihood estimation and misspecified kriging prediction under fixed domain asymptotics.
\end{abstract}

{\em Keywords:  Fixed-domain asymptotics;
  Microergodic parameter;
  Matérn Covariance;
  Maximum likelihood;
  Space-Time Generalized Wendland family; 
  Prediction.
}

\newpage
\renewcommand{\baselinestretch}{1.1}\rm

\section{Introduction}\label{sec1}

This paper provides a comparison of space-time covariance functions on the basis of their support. We target three covariance models having a full, a dynamically varying, or a compact support. The basis for comparison is their {\em compatibility}, which in turn translates into conditions for equivalence or orthogonality of Gaussian measures under given classes of kernels. 

The reason for such a comparison is motivated by the ubiquitous interest from the statistical community in comparing models with full support {\em versus} those with compact support. While such a discussion has been largely conducted for spatial domains only, we are unaware of any analog for the space-time case. Probably the main reason is the lack of space-time models with compact support. Before digging into this aspect, we contextualize the main reasons for a discussion about supports in spatial statistics. 

 The main dispute is connected with the role of compact support in mitigating the computational burden induced by huge spatial datasets while keeping a reasonable level of statistical efficiency. A wealth of statistical approaches has been proposed to circumvent this problem, and it is not within the scope of this paper to report all of them. We shall instead remind the reader of the review by \cite{popurri}, with the references therein. 
 
While being aware that there are many competitive frameworks to deal with the dilemma of {\em statistical accuracy} versus {\em computational scalability}, we believe that compact support-based approaches are quite unique in that they embrace many important aspects of space-time modeling that go beyond the mentioned trade-off. To mention a few: 
\medskip

\noindent
1. Fixed domain asymptotics has an important role in quantifying the impact of given parametric classes of covariance functions into both maximum likelihood estimation and best linear unbiased prediction, termed kriging in the spatial statistics literature \citep{Stein:1990}. The literature in this direction has been centered on the Matérn class of covariance functions \citep{stein-book, Zhang:2004}, which allows for indexing mean square differentiability of the associated Gaussian random field. Recent work \citep{BFFP} has shown that, under some mild regularity conditions, a parametric class of compactly supported covariance functions allows asymptotically achieving the same level of estimation and prediction accuracy while guaranteeing computational savings. \\[1mm]
2. Screening effect historically refers to the situation when the observations located far from the predictand (the value to be predicted at some target point) receive a small (ideally, zero) kriging weight. The screening effect phenomenon is certainly multifactorial: spatial design, the dimension of the spatial domain, and the geometric properties of the associated random field are all having an impact on the screening effect. 
Stein \cite{stein2015does} deviates from earlier literature and adopts an asymptotic approach to assess the screening effect problem. The general suggestion is to use the Matérn model to assess the screening effect, and an example in \cite{stein-book} argues against models with compact support. A recent contribution by \cite{porcu2020stein} shows that some classes of compactly supported covariance functions allow for screening effect when working in either regular or irregular settings of the spatial design. Further, their numerical studies suggest that the screening effect under compact support might be even stronger than the screening effect under a Matérn model.\\[1mm]
3. Recent findings \citep{bevilacqua2022unifying} prove that the Matérn class is a limiting case of a reparameterized version of a class of compactly supported covariance functions, called Generalized Wendland class \citep{BFFP}. This implies that the (reparametrized) Generalized Wendland model is more flexible than the Matérn model, having an additional parameter that allows for switching from compactly to globally supported covariance functions. A thorough analysis of the state of the art for the Mat{\'e}rn covariance function can be found in \cite{porcu2023mat}.  \\[1mm]
4. Tapering is a very popular technique in spatial statistics. It consists of multiplying a parametric class of covariance functions with a (normalized) class of covariance functions having additionally compact support. Here, the compact support is not estimated from data but fixed in such a way as to guarantee a desired level of sparseness for the covariance matrix (obtained from the tapered model applied to the observations). In turn, sparsity prompts a much faster inversion of the covariance matrix. This is crucial to implement both estimation and prediction. The impact of tapering on prediction has been celebrated in \cite{Furrer:2006}.  Tapering for maximum likelihood estimation under fixed domain asymptotics has been studied by \cite{du2009fixed}. A review of tapering for estimation and prediction is provided by \cite{zhang2008covariance}.

\subsection{Context and challenges}

A discussion about different supports in concert with the implications on the related covariance matrix is provided in Section \ref{sec2}. We anticipate that fully supported models have no zeros in the related covariance matrix, while dynamically supported models have zeros only in the diagonal blocks of the space-time covariance matrix. The models with compact support allow for {\em sparsity}, which means that the related covariance matrix can have many zero entries. 

Our interest in this paper is in understanding whether space-time models having either full, dynamical, or compact support might be compatible. Whenever this happens, there are precise consequences in terms of kriging efficiency under  misspecified covariance models as well as maximum likelihood estimation under fixed domain asymptotics. 

Faouzi et al. \cite{faouzi2022space} have started such a comparison between a class having full support \citep{Ryan17} against a class with dynamical supports \citep{porcu2020nonseparable}. Their results prove that compatibility is possible under suitable parametric restrictions. 

Unfortunately, there are no space-time models with compact support. An apparently simple solution is to taper a dynamically supported model with a temporal covariance model having compact support. While the validity of such construction is guaranteed by classical arguments based on properties of positive definite functions, the related spectral density - required to show compatibility conditions - is challenging. The solution to this problem requires Fourier arguments in concert with involved computations, and we defer this part to a technical Supplementary Material (see Section A therein) to avoid mathematical obfuscation. The paper centers on the conceptual exposition of the compatibility results, which are, in turn, based on the theory for equivalence of Gaussian measures.

\subsection{Contribution}

This paper provides the following contributions. We engage in the analytic closed form associated with the tapered spectral density. We then provide the asymptotic properties of the tapered spectrum. This opens the study of the parametric conditions ensuring compatibility of the involved classes of covariance functions. 

The plan of the paper is the following. Section \ref{sec2} contains a succinct and simplified mathematical background. Section~\ref{sec3} provides the proposal of this paper. Section~\ref{sec4} inspects conditions for space-time covariance compatibility. The {\em Supplementary Material} (SM throughout) is rich and contains an extended background, technical lemmas, technical results, and proofs.

\section{Background material} \label{sec2}

\subsection{Space-time covariance functions}

For the remainder of the paper, we let $d$ be a positive integer. Throughout, ${\cal D}$ is a bounded set in $\R^d$ and mimics the role of the spatial domain. Here, ${\cal T}$ is a subset of the real line and plays the role of time.  We denote by $Z=\{Z(\ss,t), (\ss,t) \in D\times {\cal T} \} $ a zero
mean Gaussian random field with index set on $D\times
{\cal T}$, with stationary covariance function $C: \R^d \times \R
\to \R$. Covariance functions are positive definite: for every arbitrary collection $\{ (\ss_i,t_l) \}$, $i=1,\ldots,N$, $l=1,\ldots, M$ and for every arbitrary finite system $\{ c_{il} \}$ of real constants, we have 
$$ \sum_{i,j=1}^N \sum_{l,m=1}^M c_{il} C \left ( \ss_i-\ss_j ,t_l - t_m \right ) c_{jm} \ge 0. $$
The paper works under the following assumptions.

\begin{tcolorbox}\label{StandingConditionA}
\begin{enumerate}
\item The covariance functions are spatially isotropic and temporally symmetric.
\item The covariance functions are continuous and absolutely integrable.
\end{enumerate}
\end{tcolorbox}

The implication of Condition~\ref{StandingConditionA}.1 is that 
$$ C(\bm{h},u) = K(\|\bm{h}\|, |u|), \qquad \bm{h} \in \R^d, \; u \in \R, $$
for some suitable function $K$ that guarantees positive definiteness. Observe that, by Bochner's theorem \citep{bochner1955harmonic}, $C$ (or equivalently, $K$) is the uniquely determined Fourier transform of a positive and bounded measure $F$. In view of Condition~\ref{StandingConditionA}.2, $F$ is absolutely continuous, and we call its derivative a {\em space-time spectral density} and use the notation $\widehat{C}$, or $\widehat{K}$, whenever there is no confusion.  Section A (SM) contains mathematical details about $\widehat{K}$, which is radial in the first argument and symmetric in the second.  

For a space-time covariance function $K$, the margins $K_{\text{S}}(\cdot)=K(\cdot,0)$ and $K_{\text{T}}(\cdot)=K(0,\cdot)$ are called spatial and temporal covariance functions, respectively. Spatial and temporal spectral densities are described in Section A (SM).

A space-time covariance function is called {\em separable} if $K(h,u)= K_{\text{S}}(h) K_{\text{T}}(u)$, where $K_{\text{S}}$ and $K_{\text{T}}$ are spatial and temporal covariance functions, respectively. In all the other cases, $K$ is called nonseparable. The function $K$ is termed compactly supported if a pair $(h_o,t_o)$ of positive real numbers exists such that $K(h,t)=0$ whenever $h \ge h_o$ and $t \ge t_o$. Consequently, the pair $(h_o,t_o)$ is called space-time compact support, and $h_o$ and $t_o$ are spatial and temporal compact supports. 

The mapping $K$ is called {\em dynamically} supported if there exists a function $\psi:[0,\infty) \to \R_+$ such that for every fixed temporal lag $t_o$ the spatial margin $K(h,t_o)$ is compactly supported with radius $\psi(t_o)$. In all the other cases, the mapping $K$ will be called {\em globally} supported. The substantial differences between the three cases are the following:  
\medskip

\noindent
1. Let $\{ (\ss_i,t_l) \}$, $i=1,\ldots,N$, $l=1,\ldots, M$ be an $N \times M$-dimensional collection of space-time points. Let $\bm{\Sigma}$ be the square $N \times M$ dimensional matrix with elements $\Sigma_{il,jm} = {\rm cov} \bigl( Z(\ss_{i},t_l), Z(\ss_j,t_m) \bigr)$. If $K$ is globally supported, then $\bm{\Sigma}$ is full, in the sense that $\bm{\Sigma}$ has no zeros. \\[1mm]
2. If $K$ is dynamically supported, then $\bm{\Sigma}$ is blockwise sparse, in the sense that the diagonal blocks in $\bm{\Sigma}$ will have as many zeros as soon as the distance between any pair of spatial points is greater than $ \psi(t_o)$. \\[1mm]
3. If $K$ is compactly supported, then $\bm{\Sigma}$ will be sparse, and it can be chosen to be as sparse as desired depending on the space-time compact support $(h_o,t_o)$.  \\

\subsection{Spatial and temporal margins}
\def\S{\text{S}}
The paper centers on two spatial isotropic covariance models, $C_{\S}$. The first model is termed Matérn class and is globally supported. The second is termed Generalized Wendland and is compactly supported. The two models have similar behavior in terms of differentiability at the origin, which makes them {\em compatible} in terms of spatial kriging prediction under fixed domain asymptotics \citep{BFFP}.
\smallskip

The Matérn class of functions, ${\cal M}(\cdot;{\alpha,\nu})$, is defined through 
\begin{equation}
\label{matern}
{\cal M}(r;{\alpha,\nu}) = \frac{2^{1-\nu}}{\Gamma(\nu)} \left ( \frac{r}{\alpha}\right )^{\nu} {\cal K}_{\nu} \left ( \frac{r}{\alpha}\right ), \qquad r \ge 0,
\end{equation}
with $\alpha >0$ a scaling parameter, and $\nu$ measuring smoothness. Here, ${\cal K}_{\nu}$ is the MacDonald function \citep{grad}. The Matérn class is a parametric class of isotropic parts of covariance functions that allow for continuously indexing the mean squared differentiability and the fractal dimension of the associated Gaussian random field. The isotropic spectral density, $\widehat{{\cal M}}$, associated with the Matérn model is reported in Section A (SM).  
\smallskip

The Generalized Wendland class ${\cal GW}(\cdot; {\beta,\mu,\kappa}):[0,\infty) \to \R$ is defined as  \citep{Gne:2002b, zast2002}
\begin{equation} \label{WG2}
{\cal GW}(r;{\beta,\mu,\kappa}):= \begin{cases}  \frac{1}{B(2\kappa,\mu+1)} \int_{r/\beta}^{1} u(u^2-(r/\beta)^2)^{\kappa-1} (1-u)^{\mu}\,\d u  ,& 0 \leq r < \beta,\\ 0,&r \geq \beta, \end{cases}
\end{equation}
where $\kappa \ge 0$, $\mu \ge (d+1)/2+\kappa $  (such a condition is required to ensure positive definiteness in $\R^d$, see \cite{Wendland:1995}) and where $\beta>0$ is the compact support parameter, and  $B$ denotes the Beta function.
The ${{\cal GW}}$ model is compactly supported on a ball embedded in $\R^d$ with radius $\beta$. The parameter $\kappa$ determines the smoothness at the origin, similar to the Matérn model. The parameter $\mu$ is a {\em convergence} parameter, where the following fact justifies this name. Let 
\begin{equation} 
\label{GW-tilde} {\widetilde{{\cal GW}}}(r; {\beta,\mu,\kappa})  = {\cal GW}\Bigg (r; \beta \Bigg(\frac{\Gamma(\mu+2\kappa +1)} {\Gamma(\mu)}\Bigg)^{\frac{1}{1+2\kappa}},\mu,\kappa \Bigg ), \qquad r \ge 0. 
\end{equation}
Arguments in \cite{bevilacqua2022unifying} prove that 
 $$\lim_{\mu\to\infty} {\widetilde{{\cal GW}}}(r; \beta, \mu, \kappa )={\cal M}(r; \beta, \kappa+1/2),\quad \kappa\geq 0,$$
with uniform convergence over the set $r \in (0,\infty)$. This result proves two facts: on the one hand, the Matérn model is a limit case of a rescaled version of the ${{\cal GW}}$ model. On the other hand, the  ${{\cal GW}}$ is a very flexible model because it allows indexing differentiability in the same fashion as the Matérn model and additionally has a parameter ($\mu$) that allows switching between compact and global supports. 

\def\T{\text{T}}
As for the temporal margins, $C_{\T}$, we shall center on the following three cases: \\
a. $C_{\T}$ belongs to the Matérn class, ${\cal M}$, as defined through (\ref{matern}); \\
b. $C_{\T}$ belongs to the Cauchy class, ${\cal C}$, defined as
\begin{equation}
\label{cauchy}
{\cal C}(t; \xi, \delta, \gamma)= \Biggl ( 1+ \Bigg ( \frac{t}{\xi} \Bigg)^{\delta}\Biggr)^{-\gamma}, \qquad t \ge 0,  
\end{equation}
where $\xi$ is a the temporal scale, and where $\gamma>0$ decides on the long memory of the process. The parameter $\delta \in (0,2]  $ decides on the local properties of the temporal trajectories in terms of fractal dimension. \\
c. $C_{\T}$ is the product of a Cauchy function, ${\cal C}$, with a special case of the ${\cal GW}$ class in (\ref{WG2}) obtained for $\kappa=0$.


\def\tt{\text{Tap}}

\section{The ${\cal DM}$, the ${\cal DGW}$, and the ${\cal DGW}_{\tt}$ families of space-time covariance functions}\label{sec3}
\begin{table}
\caption{The three space-time models used in this paper. For all of them, the parameter $\sigma^2>0$ denotes the variance of the associated process. The fourth column describes the sparsity of the associated covariance matrix $\bm{\Sigma}$. Here, \textbf{Scale}$_{\S}$ (resp. \textbf{Scale}$_{\T}$) stand for spatial or temporal scales (or compact supports), respectively. Also, \textbf{Smooth}$_{\S}$ (\textbf{Smooth}$_{\T}$) stand for the smoothness of the spatial (temporal) margin, respectively. Finally, \textbf{Nonsep} describes a nonseparability parameter. We note that the second and third rows have two additional parameters: we call the parameter $\mu$ a {\em convergence} parameter. The range of $\mu$ is defined in Condition~\ref{StandingConditionB}. The parameter $\gamma>0$ induces long or short memory (Hurst effect), and its role is not major in this paper. The parameter $\gamma$ is normally fixed and not estimated.\\[-1mm] \label{covariance}}
{
\resizebox{\columnwidth}{!}{
\begin{tabular}{|c|c|c|c|c|c|c|c|c|}
\hline
\textbf{Family}\raisebox{0mm}[5mm][4mm]{~} & $C_{\S}$ & $C_{\T}$ & $\bm{\Sigma}$ & \textbf{Scale}$_{\S}$ & \textbf{Scale}$_{\T}$ & \textbf{Smooth}$_{\S}$ &  \textbf{Smooth}$_{\T}$ & \textbf{Nonsep} \\
\hline
${\cal DM}$~ (\ref{DM})\raisebox{0mm}[5mm][4mm]{~} & ${\cal M}$~ (\ref{matern}) & ${\cal M}$~ (\ref{matern}) & {\sc Full} & $\zeta>0$ & $\upsilon>0$ & $\nu>0$ & $\nu>0$ & $\epsilon \in [0,1]$ \\
\hline
${\cal DGW}$~ (\ref{WG21}) \raisebox{0mm}[5mm][0mm]{~}& ${\cal GW}$ (\ref{WG2}) & ${\cal C}$~ (\ref{cauchy}) & {\sc Bloc} & $\beta>0$ & $\xi>0$ & $\kappa \ge 0$ & $\delta \in (0,2)$ & {\sc None} \\
 &  &  & {\sc Diag Sparse} &  & & & &  \\
\hline
${\cal DGW}_{\tt}$ (\ref{WG21})\raisebox{0mm}[5mm][4mm]{~} & ${\cal GW}$ (\ref{WG2}) & ${\cal C} \times {\cal GW}$  &{\sc Sparse} & $\beta>0$ & $\xi>0$ & $\kappa \ge 0$ & $0$ & {\sc None} \\
\hline
\end{tabular}} }
\end{table}

This paper deals with three parametric classes of nonseparable covariance functions. The amount of parameters and the complexities of the algebraic form need some explanation, and Table \ref{covariance} helps to do the job. 

Some comments are in order. There is no free lunch: these three families have features and drawbacks. The ${\cal DM}$ family has a parameter, $\epsilon$, that allows to continuously index nonseparability. The separable case is attained for $\epsilon=0$. The spatial and temporal margins are both of Matérn type. Unfortunately, no algebraically closed forms are available except for the separable case. Another drawback is scalability: the associated covariance matrix $\bm{\Sigma}$ is full. On the other hand, both ${\cal DGW}$ and ${\cal DGW}_{\tt}$ are available in algebraically closed forms. Their associated covariance matrices $\bm{\Sigma}$ are block diagonal sparse, and sparse, respectively. For the latter, any level of sparsity (the percentage of zeros in the matrix $\bm{\Sigma}$) can be achieved by fixing $\xi$ {\em ad hoc}. Another relevant comment is that the temporal margin $C_{\T}$ is either nondifferentiable or infinitely differentiable under the  ${\cal DGW}$ model, while is it always nondifferentiable at the origin for the ${\cal DGW}_{\tt}$ model. A final comment is that both ${\cal DGW}$ and ${\cal DGW}_{\tt}$ models have no parameter that allows switching from separability to nonseparability. 

Our effort here goes in the direction of whether we can increase the sparsity of the covariance matrix $\bm{\Sigma}$ from ${\cal DM} $ to ${\cal DGW}_{\tt}$ at little loss in terms of statistical accuracy. This is measured through the {\em compatibility} of these models under fixed domain asymptotics. 
We are now ready to formally define the three models: 
\medskip

\noindent \emph{1. The space-time Matérn model}, denoted ${\cal DM}$.\\ 
Define {\sc FT}$_{\T}$ to be the Fourier transform with respect to the temporal component. Mathematical details are contained in the  Section A (SM). Let $g_{{\cal M}}(\cdot,\cdot; \bm{\theta}): [0,\infty)^2 \to \R $ be the function defined at (A.3) in  Section A (SM). Specifically, $g_{{\cal M}}$ depends on the spatial distance in the first coordinate and on the {\em time frequency} in the second.  The parameter vector contains the parameters described in the first row of Table~\ref{covariance}. Hence, $\btheta=(\nu,\xi,\upsilon,\epsilon,\sigma)^{\top}$ with
$\top$ denoting the transpose of a vector.  Section A (SM) explains how the partial Fourier transform with respect to time, $F_{\T|\S}$ as defined through (A.2) in SM provides a covariance function
\begin{equation}\label{DM}
{\cal DM}(r,t;{\btheta})= \text{FT}_{\T|\S} \Big [ g_{{\cal M}}(\cdot, \tau; \bm{\theta}) \Big ] (r,t), \qquad  (r,t)\in[0,\infty),
\end{equation}
 that depends on the spatial and temporal distances $(r,t)$. The fact that it is a valid space-time covariance function is proved by \cite{fuentes2008class}. The spatial and temporal margins of this covariance function are both of the Matérn types. \\
 It is worth remarking that the function $g_{\bm{\theta}}$ has a factor $\ell(\bm{\theta})$, having a mathematically involved expression. To avoid mathematical obfuscation, we report its expression in Equation (A.4) in  Section A (SM). 
\smallskip

%
%

\noindent \emph{2. Dynamical Generalized Wendland model}, denoted ${\cal DGW}$. \\
Using the same notation as much as in \cite{porcu2020nonseparable}, we now introduce the ${\cal DGW} $ class of space-time covariance functions, defined as
\begin{equation} \label{WG21}
{\cal DGW}(r,t; {\bX})=\sigma^2 {\cal C} \left ( t; \xi,\delta,\gamma \right )    {\cal GW} \Bigg ( \frac{r}{{\cal C} \left ( t; \xi,\delta,{1} \right ) }; {\beta,\mu,\kappa} \Bigg ), \qquad r,t \ge 0,
\end{equation}
The meaning of each parameter is explained in the second row of Table \ref{covariance}. Hence, the resulting parameter $\bX$ amounts to $\bX= (\sigma^2, \beta, \mu, \kappa, \xi,\delta, \gamma)^{\top}$. \\
We note that the parametric conditions ensuring the validity of the ${\cal DGW}$ model are the ones in Table \ref{covariance}. Yet, some technical additional conditions are required, and we formalize them below.

\begin{tcolorbox}\label{StandingConditionB}  Let  $\eta=({d+1})/{2}+\kappa$. For the remainder of the paper, we always work under the parametric condition 
\begin{equation} \label{standing_condition_B}
    \mu > \max \Big ( (d+5)/2 + \kappa + \varsigma^{*}, \eta + (d+1)/2 \Big ) \qquad \text{and} \qquad \gamma \ge \max \left (  (d+3)/2 + 2 \kappa, 2 \kappa +3 \right ).
\end{equation}
The stricty positive constant $\varsigma^{*}$ plays no role in this paper, and we omit its specification while referring to Table 2 and Section 3.4 in \cite{porcu2020nonseparable} for details.


\end{tcolorbox}

\noindent{\em 3. Tapered dynamical Generalized Wendland model}, denoted ${\cal DGW}_{\tt}$. \\
This model is obtained by temporal tapering of the ${\cal DGW}$ model. The resulting equation is 
\begin{equation}
\label{WG22}
{\cal DGW}_{\tt}(r,t; \bX) = {\cal DGW}(r,t; \bX) \times {\cal GW} (t; \xi,4,0) , \qquad r,t \ge 0.
\end{equation}
We note that $$ {\cal GW} (t; \xi,4,0) = \biggl( 1- \frac{t}{\xi}\bigg)_{+}^4, \qquad t \ge 0. $$
The exponent $4$ is the minimal exponent that guarantees our theoretical results hold. The results can be shown for any exponent that is greater or equal to $4$, at the expense of additional notation. We keep things simple here.\\
Clearly, the construction (\ref{WG22}) is positive definite as it is the product of two positive definite functions. The resulting covariance matrix, $\bm{\Sigma}_{\tt}$, is the Kronecker product of the matrix $\bm{\Sigma}$, associated with ${\cal DGW}$, with the temporal matrix $\bm{\Sigma}_{\tt}$, associated with ${\cal GW}(\cdot; \xi, 4,0)$. On the other hand, the spectral density associated with ${\cal DGW}_{\tt}$ has a very complicated expression, being the convolution (over time) of the space-time spectral density associated with ${\cal DGW}$ with temporal spectral density associated with ${\cal GW}(\cdot; \xi, 4,0)$. \\
Another relevant comment is the following: let 
$$ {\cal D}\widetilde{{\cal GW}}(r,t; \bX) = \sigma^2 {\cal C} \left ( t; \xi,\delta,\gamma \right )    \widetilde{{\cal GW}} \Bigg ( \frac{r}{{\cal C} \left ( t; \xi,\delta,\gamma \right ) }; {\beta,\mu,\kappa} \Bigg ), \qquad r,t \ge 0. $$
Then, convergence arguments as much as in 
\cite{bevilacqua2022unifying} prove that 
$$ \lim_{\mu \to \infty} {\cal D}\widetilde{{\cal GW}}(r,t; \bX) = \sigma^2 {\cal C} \left ( t; \xi,\delta,\gamma \right )    {\cal M} \Bigg ( \frac{r}{{\cal C} \left ( t; \xi,\delta,\gamma \right ) }; {\beta,\kappa+1/2} \Bigg ), \qquad r,t \ge 0,  $$
uniformly for all $r,t$. The right-hand side of the above equation responds to the Gneiting class of covariance functions \citep{gneiting2002nonseparable}, justifying calling $\mu$ a convergence parameter (Table~\ref{covariance}). 

We note that the parametric ranges ensuring positive definiteness of ${\cal DGW}$ and ${\cal DGW}_{\tt}$ are reported in  Section A (SM).

\section{Compatibility theorems for space-time covariance functions} \label{sec4}

We illustrate compatibility here, and for a formal introduction, the reader is referred to Section A in SM.

Equivalence and orthogonality of probability measures are useful tools when assessing the asymptotic properties of both prediction and estimation for stochastic processes.
Denote with $P_i$, $i=0,1$, two probability measures defined on the same measurable space $\{\Omega, \cal F\}$. $P_0$ and $P_1$ are called equivalent (denoted $P_0 \equiv P_1$) if $P_1(A)=1$ for any $A\in \cal F$ implies $P_0(A)=1$ and vice versa.
On the other hand,  $P_0$ and $P_1$ are orthogonal (denoted $P_0 \perp P_1$) if there exists an event $A$ such that $P_1(A)=1$ but $P_0(A)=0$. For a stochastic process $Z=\{Z(\ss,t), (\ss,t) \in \R^d\times\R\}$, to define previous concepts, we restrict the event $A$ to the $\sigma$-algebra generated by $\{Z(\ss,t), (\ss,t)\in D\times~{\cal T}\}$ where $D\times~{\cal T} \subset \R^d\times\R$. We emphasize this restriction by saying that the two measures are equivalent on the paths of $Z$.

Gaussian measures are completely characterized by their mean and covariance functions.
We write $P(C)$ for a Gaussian measure with zero mean and covariance function $C$.
It is well known that two Gaussian measures  are either equivalent or orthogonal on the paths of $Z$ \citep{Ibragimov-Rozanov:1978}. Conditions for equivalence of Gaussian measures are explained in Section A (SM).

The information above allows for providing the following formal statement: 

\begin{rem}
Two space-time covariance functions are called compatible if their induced Gaussian measures are equivalent on the paths of $Z(\ss,t), (\ss,t) \in D\times{\cal T}$.
\end{rem} 


Some notation is now necessary. We define 
$\bX_i= (\sigma_i^2, \beta_i, \mu, \kappa_i, \xi_i,\delta_i, \gamma_i)^{\top}$, $i=0,1$.

\noindent {Throughout, we always suppose that Condition \ref{standing_condition_B} holds. We avoid stating this explicitly in every formal statement for the sake of simplicity.}
\begin{thm}\label{TW_vs_TW}
Let $\eta_i= ({d+1})/{2}+\kappa_i$, $i=0,1$ { and $\gamma_0=\gamma_1=\gamma$}.
\\
Let $$ a_{1,i}=5 \sqrt{\frac{2}{\pi}}\frac{\Gamma(\mu+2\eta_i)}{\Gamma(\mu)} \qquad \text{and} \qquad   a_{2,i}=a_{1,i} \frac{\sqrt{\pi}  \; (\gamma+d-2\eta_i)\Gamma({\delta_i}/{2}+1)}{5 \Gamma(-{\delta_i}/{2})\Gamma(6)^2}, $$ with 
$$ L_i= \frac{\Gamma(\kappa_i +(d+1)/2 )  \Gamma(2 \kappa_i +\mu+1) }{\pi^{d/2} \Gamma(\kappa_i +1/2) \Gamma(\mu+ 2 \eta_i) }. $$
Consider the covariance models: ${\cal DGW}_{\tt}(\cdot,\cdot; \bX_{0})$ and  ${\cal DGW}_{\tt}(\cdot,\cdot; \bX_{1}$). For a given $\kappa_1=\kappa_0$ and {$\delta_i\in\{1,2\}$}, and for any
bounded infinite set $D\subset \R^d$, $d=1, 2$,
 the following compatibility assertions are true: 
\begin{enumerate}
\item for  $\delta_0=\delta_1=1$ and $\kappa_0=\kappa_1$,  ${\cal DGW}_{\tt}(\cdot,\cdot; {\bX}_0)$ and ${\cal DGW}_{\tt}(\cdot,\cdot; {\bX}_1)$ are compatible if and only if 
 $$
L_1 \left(\frac{4a_{2,1}}{\xi_1^{2}}+\frac{a_{1,1}}{\xi_1}\right)\frac{\sigma_1^2  }{\beta_1^{2 \kappa_1+1}}= L_0 \left(\frac{4a_{2,0}}{\xi_0^{2}}+\frac{a_{1,0}}{\xi_0} \right)\frac{\sigma_0^2 }{\beta_0^{2 \kappa_0+1}}.
$$
\item for  $\delta_0=\delta_1=2$ and $\kappa_0=\kappa_1$,  ${\cal DGW}_{\tt}(\cdot,\cdot; {\bX}_0)$ and ${\cal DGW}_{\tt}(\cdot,\cdot; {\bX}_1)$ are compatible if and only if $$
\sigma_1^2  L_1\frac{\beta_1^{d-2\eta_1}}{\xi_1}= \sigma_0^2  L_0\frac{\beta_0^{d-2\eta_0}}{\xi_0}.$$
\end{enumerate}
\end{thm}

The following result relates to the compatibility of the ${\cal DGW}$ model with parameter vector $\bX_0$ with the corresponding tapered version ${\cal DGW}_{\tt}$ and a misspecified parameter vector $\bX_1$. All the constants appearing in the result below are available in theorems given in Theorem B.1 in SM.

\begin{thm}\label{TAPvsDW}
For  given  $\delta_1=1$, $\gamma_0>2\kappa_0+3$ and  $\gamma_1 >2\kappa_1+3$. Consider the models ${\cal DGW}_{\tt}(\cdot,\cdot; \bX_1)$ and ${\cal DGW}(\cdot,\cdot; \bX_0)$. 
{Let
$$\varrho_{\gamma_0,\eta}=\frac{(d+\gamma_0-2\eta)\Gamma(\delta_0+1)\sin(\frac{\pi\delta_0}{2})}{\xi_0^{\delta_0}\pi}.$$ }
Then, the compatibility of the two models for any
bounded infinite set $D\times {\cal T}\subset \R^d\times\R$, $d=1,2$, holds  if {$\delta_0=1+2(\kappa_1-\kappa_0)$ } and
$$ \varrho_{\gamma_0,\eta} \sqrt{\frac{\pi}{2}} \frac{a_{1,0} \; L_0}{5}    \; \frac{\sigma_0^2 }{\beta_0^{-(1+2\kappa_0)}}=\frac{a_{1,1} L_1}{\sqrt{2\pi}} \; \frac{\sigma_1^2 }{ \beta_1^{-(1+2\kappa_1)}}, $$
with all the relevant constants as being defined through Theorem \ref{TW_vs_TW}. 
 
\end{thm}
The final result relates the compatibility of the tapered version of the ${\cal DGW}$ model to the space-time Matérn model ${\cal DM}$. Some additional notation is necessary. We call 
$\btheta_0=(\nu,\zeta,\upsilon,\epsilon,\sigma_0)^{\top}$ from the parameterization related to the ${\cal DM}$ model. 


\begin{thm}\label{TGWDvsMat}
For  given $\nu>(d+1)/{2}$ and 
 $\epsilon\in(0,1]$ consider the models ${\cal DM}(\cdot,\cdot; \btheta_0)$ and
${\cal DGW}_{\tt}(\cdot,\cdot; {\bX}))$. Let $\ell(\btheta_0)$ be as defined at (A.4) in Section A (SM).
If
$$\Bigg [\frac{\sigma^2  L}{\sqrt{2\pi}\beta^{1+2\kappa}}\left(\frac{4a_2}{\xi^{2}}+\frac{a_{1}}{\xi_1}\right)\Bigg ]\mathbf{1}_{\delta=1}=\ell(\btheta_0)\epsilon^{-2\nu},$$ 
with   $2\nu=\eta+1,$ and $\kappa>{d-1}/{2}$ then,  for any
bounded infinite set $D\times {\cal T}\subset \R^d\times\R$, $d=1,2$, the two models are compatible.
\end{thm}



These compatibility theorems have severe implications for both ML estimation and kriging prediction under fixed domain asymptotics. 

Following the arguments in \cite{Zhang:2004}, an immediate consequence of Theorem \ref{TW_vs_TW} is that for fixed  $\delta$, $\mu$  and $\gamma$ the parameters $\sigma^2$, $\beta$, $\kappa$ and $\xi$   cannot be estimated consistently. Instead, {for $\btau= ( \beta,\xi)^{\top}$}, the microergodic parameter
$$ \left(\frac{4a_{2}}{\xi^{2}}+\frac{a_{1}}{\xi_1}\right)\frac{\sigma^2(\btau)L}{\beta^{2 \kappa+1}}$$ is consistently estimable. The proof comes straight by using the same arguments as in \cite{faouzi2022space}, and is hence omitted. 


The second implication is in terms of kriging prediction. When working under fixed domain asymptotics, the kriging variance under the misspecified model will tend to be equal to the kriging variance under the correct model. Proofs are obtained by mimicking \cite{BFFP} and hence excluded from this paper. 


Several parameters in the ${\cal DGW}_{\tt}$ model framework affect the sparsity of the resulting covariance matrix, as shown in Table~\ref{covariance}. Together with the compatibility theorems, this enables handling very large datasets, despite the non-trivial form of the Generalized Wendland class. (As demonstrated in \citep{GeneralizedWendland}, fast implementations of the covariance functions exist.)
Similar to the spatial setting, the best results are achieved for both estimation and prediction settings when balancing the smoothness and decay of the misspecified covariance function. Still, there is a nontrivial interplay of the parameters and the behavior at the origin (see \url{http://shiny.math.uzh.ch/user/furrer/shinyas/GenWendTap/}). 
The optimal parameter choices strongly depend on the specific dataset, making it unwise to provide recommendations for the choices thereof.

This paper has addressed a crucial gap in the infill-asymptotics framework by providing compatibility theorems for arbitrary compact support in both space and time.

\section{Conclusion}

Our effort has allowed us to quantify statistical accuracy in terms of estimation and prediction when tapering the temporal part of a covariance function that is spatially dynamically supported. Arguments in \cite{furrer2006covariance} suggest that better performance in terms of prediction accuracy might be achieved by using temporal tapers that are smoother at the origin. Yet, the convoluted space-time spectral density becomes analytically intractable even in the case of a temporal taper that is once differentiable at the origin. Since this paper is more oriented to analytical solutions rather than heuristics, we prefer to leave the issue of improving temporal differentiability as an open problem. \\
Future developments include multivariate random fields, for which the covariance is a matrix-valued function, and functional processes, for which the theory of tapering for covariance operators is still elusive. \\
Recent applications in statistics and machine learning include processes that are defined over manifolds (for instance, the sphere) cross time. We are not aware of any theoretical study regarding tapering for this case, and a big effort is needed, starting with the theory on the equivalence of Gaussian measures for processes defined over spheres cross time.

\section*{Acknowledgements}

Tarik Faouzi is supported  by FONDECYT grant 11200749, Chile.
Reinhard Furrer was supported by the Swiss National Science Foundation SNSF-175529.
Emilio Porcu is supported by FSU grant number FSU-2021-016.

\section*{Acknowledgement}

Emilio Porcu is supported by FSU grant number FSU-2021-016. Tarik Faouzi is supported  by FONDECYT grant 11200749, Chile.

\appendix

\section{Appendix}

\subsection{Background Material} \label{App_back}

The results in  \cite{Porcu:Gregori:Mateu:2006}  show that a space-time covariance function $K$ admits a uniquely determined representation in terms of Hankel transforms of probability measures, namely 
$$ K(r,t)= \int_0 ^{\infty}\int_{0}^{\infty} \Omega_{d}(r \xi_1) \cos(t\xi_2) F(\d( \xi_1,\xi_2)), \qquad t \geq 0, r \geq 0,$$
where $\Omega_{d}(t)=t^{-({d-2})/{2}}J_{(d-2)/2}(t)$ and $J_{\nu}$
is the Bessel function of the first kind of order $\nu>0$. Classical
Fourier inversion arguments show that, if $K$ is absolutely integrable, then its isotropic spectral density, $\widehat{K}$, can be written as a space-time Fourier transform, FT$_{\text{ST}}$, that is 
\begin{equation}\label{FT}
 \widehat{K}(z,\tau)= \text{FT}_{\text{ST}} \Big [ K(r,t)\Big ](z,\tau)=\int_{\R^d} \int_{\R}e^{- i  |u|\tau -i \|\br\|\|\bz\|} K(\|\br\|,|u|) {\rm d} \br {\rm d} u, \qquad z,r \ge 0.   
\end{equation}
The spatial margin $K_{\text{S}}(r) = K(r,0)$ will have a spectral density determined by 
\begin{equation}\label{FT}
\widehat{K}_{\text{S}}(z)=\text{{\sc FT}}_{\text{S}} \Big [ K_{\S}(r)\Big ](z)= \int_{\R^d} e^{-i \|\br\|\|\bz\|} K(\|\br\|,0) {\rm d} \br , \qquad z \ge 0.   
\end{equation}
A similar comment applies to the temporal Fourier transform FT$_{\text{T}}$. \\
The Mat{\'e}rn class ${\cal M}$ as defined through (\ref{matern}) has  Fourier transform in $\R^d$ that is identically equal to  $$\widehat{\cal M}(z; \alpha,\nu) = \text{FT}_{\text{S}}\Big [ {\cal M}(r; \alpha, \nu )\Big ](z) = \frac{\Gamma(\nu+d/2)}{\pi^{d/2} \Gamma(\nu)}
\frac{ \alpha^d}{(1+\alpha^2z^2)^{\nu+d/2}}
, \; \; z \ge 0.  $$
Partial Fourier transforms of functions defined over product spaces are possible as well. Let $g: [0,\infty)^2 \to \R$ be a continuous function such that $C_{\text{S}}(\cdot;\tau):= g(\cdot,\tau)$ is an isotropic spatial covariance function for all $\tau \ge 0$ and such that $g(r, \cdot)$ is continuous, positive and integrable on the positive real line for all $r \ge 0$. Then, the mapping
\begin{equation}
    \label{FT_partial} \text{FT}_{\text{T}|\text{S}} \Big [ g(\cdot, \tau)\Big ] (r,t): = 2 \int_{[0,\infty)]} \cos \left ( t \tau \right ) g(r, \tau) {\rm d} \tau, \qquad r,t \ge 0,
\end{equation} 
is an isotropic space-time covariance function. This statement is discussed in \cite{cressie-huang}, with some examples that have been disproved by \cite{gneiting1}. A formal statement is provided by \cite{zast-porcu}, but the principle has been used earlier by \cite{fuentes2008class} to build the 
 ${\cal DM}$ class,  defined as 
\begin{equation*}
K_{\cal DM}(r,t;{\btheta})=\text{FT}_{\text{T}|\text{S}} \Big [ g_{\cal M}(\cdot, \tau; \btheta)\Big ] (r,t), \qquad  (r,t)\in[0,\infty),
\end{equation*}
where 
 \begin{equation} \label{gm}
 \displaystyle~g_{{\cal M}}(r,\tau)=\frac{\sigma^2\ell(\btheta)\pi^{d/2}}{2^{\nu-d/2-1}}\left(\frac{r}{a(\tau)}\right)^{\nu-d/2} \left(\upsilon^2+\epsilon \tau^2
  \right)^{-\nu} {\cal K}_{\nu-d/2} \left(a(\tau)r \right),
\end{equation}
 with  $a(\tau)=\sqrt{{\zeta^2(\upsilon^2+\tau^2)}/{\upsilon^2+\epsilon \tau^2}}$, $\tau \ge 0$ and $\Gamma$ the Gamma function \citep{grad}.
As for the function  $\ell$, which depends on $\btheta$, it is defined as
\begin{equation} \label{NN}
  \ell(\btheta)=
\frac{1}{\upsilon^{1-2\nu}\Gamma(\nu-\frac{1}{2})\sqrt{\pi}}\Bigl(\int_{\R^d}(\zeta^2+\norm{\br}^2)^{\frac{1}{2}-\nu}\left(\zeta^2+\epsilon\norm{\br}^2\right)^{-\frac{1}{2}}\text{d}\br\Bigr)^{-1},
\end{equation}
which can be expressed as
\begin{align}
  \int_{\R^d}(\zeta^2+&\norm{\br}^2)^{\frac{1}{2}-\nu}\left(\zeta^2+\epsilon\norm{\br}^2\right)^{-\frac{1}{2}}\text{d}\br \nonumber \\
&=\frac{2\pi^{d/2}}{\Gamma(d/2)}\int_{0}^{\infty}(\zeta^2+r^2)^{\frac{1}{2}-\nu}\left(\zeta^2+\epsilon r^2\right)^{-\frac{1}{2}}\text{d}r \nonumber \\
&=\frac{\zeta^{-2\nu}\pi^{d/2}}{\Gamma(d/2)}\Bigg[ \zeta^{d/2}\Gamma(d/2)\Gamma(\nu-d/2-1/2)\mathstrut_2 F_1(\frac{1}{2},\frac{d}{2};\frac{3+d-2\nu}{2};\epsilon) \nonumber \\
&\quad+ \zeta^{d/4}\epsilon^{\nu-d/2-1/2}\Gamma(\frac{1+d-2\nu}{2})\Gamma(\nu-\frac{d}{2})\mathstrut_2 F_1(\nu-\frac{1}{2},\nu-\frac{d}{2};\nu-\frac{d}{2}+\frac{1}{2};\epsilon)\Bigg].
\label{newrep}
\end{align} 
For the cases $\epsilon=0$ and $\epsilon=1$, Equation~\eqref{newrep} simplifies and we have for Equation~\eqref{NN} 
\begin{equation*}
  \ell(\btheta)=\frac{\zeta^{2\nu-d}\upsilon^{2\nu-1}}{\Gamma(\nu-\frac{d+1}{2})}
\qquad\text{and}\qquad
  \ell(\btheta)=\frac{\zeta^{2\nu-d}\upsilon^{2\nu-1}\Gamma(\nu)}{\Gamma(\nu-\frac{d}{2})\Gamma(\nu-\frac{1}{2})},
\end{equation*}
respectively.

\subsection{Equivalence of Gaussian measures} \label{App_Comp}

Equivalence and orthogonality of probability measures are useful tools when assessing the asymptotic properties of both prediction and estimation for stochastic processes.
Denote with $P_i$, $i=0,1$, two probability measures defined on the same
 measurable space $\{\Omega, \cal F\}$. $P_0$ and $P_1$ are called equivalent (denoted $P_0 \equiv P_1$) if $P_1(A)=1$ for any $A\in \cal F$ implies $P_0(A)=1$ and vice versa.
 On the other hand,  $P_0$ and $P_1$ are orthogonal (denoted $P_0 \perp P_1$) if there exists an event $A$ such that $P_1(A)=1$ but $P_0(A)=0$. For a stochastic process $Z=\{Z(\ss,t), (\ss,t) \in \R^d\times\R\}$, to define previous concepts, we restrict the event $A$ to the $\sigma$-algebra generated by $\{Z(\ss,t), (\ss,t)\in D\times~{\cal T}\}$ where $D\times~{\cal T} \subset \R^d\times\R$. We emphasize this restriction by saying that the two measures are equivalent on the paths of $Z$.

Gaussian measures are completely characterized by their mean and covariance function.
We write $P(C)$ for a Gaussian measure with zero mean and covariance function $C$.
It is well known that two Gaussian measures  are either equivalent or orthogonal on the paths of $Z$ \citep{Ibragimov-Rozanov:1978}.

Let $d$ be a positive integer. Let $P(C_i)$, $i=0, 1$ be two zero
mean Gaussian measures associated with a RF $Z$ defined over a
bounded set ${\cal D}\times{\cal T}$ of $\R^d\times\R$, with
covariance function  $C_i$ such that $C_i=\sigma_i^2 K_i$,
for $K_i$ being its isotropic part, and associated spectral density
$\widehat{C}_i(\bz,\tau)=\sigma_i^2 \widehat{K}_i(\norm{\bz},|\tau|)$,
with $\widehat{K}_i$ as in~\eqref{FT}. Using results in
\cite{Sko:ya:1973},  \cite{Ibragimov-Rozanov:1978} and
\cite{Stein:2004}, \cite{Ryan17} have shown that, if for some $a>0$,
$\widehat{C}_0(\bz,\tau)\norm{(\bz,\tau)}^a$ is bounded away from
0 and $\infty$ as $(\norm{\bz,\tau}) \to \infty$, where
$(\norm{\bz,\tau})=\sum_{i=1}^d~z_i^2+\tau^2$ and
for some finite and positive~$c$,\\
%
\begin{equation}\label{spectralfinite2_bis}
\int_{\cal A} z^{d-1}\;\left\{ \frac{\widehat{K}_1(z,\tau)-\widehat{K}_0(z,\tau)}{\widehat{K}_0(z,\tau)} \right\}^2\;
{\rm d} z {\rm d} \tau <\infty,
\end{equation}
where  ${\cal A}={\cal A}_1\cup{\cal A}_2\cup{\cal A}_3\cup{\cal
A}_4\cup{\cal A}_5$ with   ${\cal A}_1=\{z>c, \tau>c\}; {\cal
A}_2=\{z>c,0\leq\tau<c\};{\cal A}_3=\{0\leq z<c, \tau>c\}; {\cal
A}_4=\{z>c_1, 0\leq\tau<c\}; {\cal A}_5=\{0\leq z<c, \tau>c_2\}$,
and $c_1$, $c_2$ be two constants satisfying $0 < c_1, c_2 < c$ such
that $c^2_1+ c^2_2< c^2$ \citep{Ryan17,faouzi2022space}. Then, for any
bounded subset $D\times{\cal T}\subset \R^d\times\R$,
$P(C_0)\equiv P(C_1)$ on the paths of $Z(\ss,t), (\ss,t) \in D\times{\cal T}$.

The information below allows to provide the following formal statement: 

\begin{tcolorbox}[title = Formal Statement A]
Two space-time covariance functions are called compatible if their induced Gaussian measures are equivalent on the paths of $Z(\ss,t), (\ss,t) \in D\times{\cal T}$.
\end{tcolorbox}

\section{Technical Results} \label{App_Technical}

We recall the generalized hypergeometric functions  \citep{Abra:Steg:70}   $\mathstrut_p F_q$:
\begin{equation*}
\mathstrut_p F_q(a;b,c;z)=\sum_{k=0}^{\infty}\frac{(a_1)_{k}\cdots (a_p)_{k} z^{k}}{(b_1)_{k}\cdots (b_q)_{k}k!}, \qquad z \in\mathbb{R},
\end{equation*}
with
 \[ (a)_{k}=
  \begin{cases}
   \frac{\Gamma(k+a)}{\Gamma(a)}
     & \quad \text{for } k\geq1,\\
    1  & \quad \text{for } k=0,
  \end{cases}
\] being the Pochhammer symbol. The following theorem expresses the spectral density of the taper function as an infinite series of $\mathstrut_2 F_3$ functions.

\subsection{Technical Lemmas}

\begin{lem}\citep{gonzalez2018closed} \label{lem2}
Let $a_i, i=1,2,3$  and  $b_j, j=1,2,3,4$  be a real parameters such that $a_j\neq a_i$ for $i\neq j$. Then, for $z \to \infty$, the generalized hypergeometric function $\mathstrut_3F_4$ has the following expression:
\begin{equation}
\begin{aligned}
&\mathstrut_3 F_4\left(a_1,a_2,a_3;b_1,b_2,b_3,b_4;-\Big(\frac{z}{2}\Big)^{2}\right)= \\
&\frac{\prod_{j=1}^4 \Gamma(b_j)}{\sqrt{\pi}\prod_{j=1}^3 \Gamma(a_j)} \left(\frac{z}{2}\right)^{2\varrho}\Bigg[\cos(\pi\varrho+z)\left(1+O(z^{-2})\right)+\frac{\varrho_1}{z}\sin(\pi\varrho+z)\left(1+O(z^{-2})\right)\Bigg]+\\
&\frac{\prod_{j=1}^4 \Gamma(b_j)}{\prod_{j=1}^{3}\Gamma(a_j)}\sum_{k=1}^3 \frac{\Gamma(a_k)\prod_{j=1,j\neq~k}^3\Gamma(a_j-a_k)}{\prod_{j=1}^4\Gamma(b_j-a_k)} \left(\frac{z}{2}\right)^{-2a_k}\left(1+O(z^{-2})\right),\\
\end{aligned}
\end{equation}
with
$\varrho_1=2\left(\sum_{j=2}^4\sum_{i=1}^{j-1}b_j b_i-\sum_{j=2}^3\sum_{i=1}^{j-1}a_j a_i+\frac{1}{4}(3\sum_{i=1}^3 a_i+\sum_{i=1}^4 b_i-2)(\sum_{i=1}^3 a_i-\sum_{i=1}^4 b_i)-\frac{3}{16}\right)$ and $\varrho=\frac{1}{2}(\sum_{i=1}^3 a_i-\sum_{i=1}^4 b_i+\frac{1}{2})$.
\end{lem}

\begin{lem}\label{lem3}
Let $n,d \in \mathbb{N}$ and $\gamma>0$. Then, the Pochhammer symbol $(2n+d+\gamma)_k$ has the following expression:
$$(2n+d+\gamma)_k=\frac{(\frac{d+k+\gamma}{2})_n (\frac{d+k+1+\gamma}{2})_n (d+\gamma)_k}{(\frac{d+\gamma}{2})_n(\frac{d+1+\gamma}{2})_n}.$$
\end{lem}
\begin{proof}
We invoke the duplication formula $$2^{1-2x}\sqrt{\pi}\Gamma(2x)=\Gamma(x)\Gamma(x+1/2),$$ to get that the Pochhammer function $(2n+d+\gamma)_k$ has the following expression
\begin{equation*}
(2n+d+\gamma)_k=2^k \Big (n+\frac{d+\gamma}{2} \Big)_\frac{k}{2} \Big (n+\frac{d+\gamma+1}{2} \Big )_\frac{k}{2},
\end{equation*}
with  $(x)_{\frac{k}{2}}={\Gamma(x+k/2)}/{\Gamma(x)}$.\\
Next,
\begin{equation}\label{AB1}
\Big (n+\frac{d+\gamma}{2} \Big)_\frac{k}{2}=\frac{\Gamma(n+\frac{d+\gamma}{2}+\frac{k}{2})}{\Gamma(n+\frac{d+\gamma}{2})}=\frac{(\frac{d+\gamma}{2}+\frac{k}{2})_n\Gamma(\frac{d+\gamma}{2}+\frac{k}{2})}{(\frac{d+\gamma}{2})_n\Gamma(\frac{d+\gamma}{2})}
\end{equation}
and
\begin{equation}\label{AB2}
\Big (n+\frac{d+\gamma+1}{2} \Big)_\frac{k}{2}=\frac{\Gamma(n+\frac{d+\gamma+1}{2}+\frac{k}{2})}{\Gamma(n+\frac{d+\gamma+1}{2})}=\frac{(\frac{d+\gamma+1}{2}+\frac{k}{2})_n\Gamma(\frac{d+\gamma+1}{2}+\frac{k}{2})}{(\frac{d+\gamma+1}{2})_n\Gamma(\frac{d+\gamma+1}{2})}.
\end{equation}
Using Equations (\ref{AB1}) and (\ref{AB2}), we write
\begin{equation}
\begin{aligned}
(2n+d+\gamma)_k=&2^k\frac{(\frac{d+\gamma+1}{2}+\frac{k}{2})_n\Gamma(\frac{d+\gamma+1}{2}+\frac{k}{2})}{(\frac{d+\gamma+1}{2})_n\Gamma(\frac{d+\gamma+1}{2})}\frac{(\frac{d+\gamma}{2}+\frac{k}{2})_n\Gamma(\frac{d+\gamma}{2}+\frac{k}{2})}{(\frac{d+\gamma}{2})_n\Gamma(\frac{d+\gamma}{2})}\\
&=\frac{(\frac{d+k+\gamma}{2})_n (\frac{d+k+1+\gamma}{2})_n (d+\gamma)_k}{(\frac{d+\gamma}{2})_n(\frac{d+1+\gamma}{2})_n}.\\
\end{aligned}
\end{equation}
The proof is completed.
\end{proof}

\begin{lem} \citep{nijimbere2017evaluation} \label{lem1}
Let $a_i, i=1,2$  and  $b_j, j=1,2,3$  be a real parameters with $a_1\neq a_2$. Then, for $\tau \to \infty$, 
\begin{eqnarray}
\mathstrut_2 F_3\left(a_1,a_2;b_1,b_2,b_3;-\Big(\frac{\tau}{2}\Big)^{2}\right)\sim &
\frac{\Gamma(b_1)\Gamma(b_2)\Gamma(b_3)\Gamma(a_1-a_2)}{\Gamma(a_2)\Gamma(b_1-a_1)\Gamma(b_2-a_1)\Gamma(b_3-a_1)}(\tau/2)^{-2a_1}+\\
&\frac{\Gamma(b_1)\Gamma(b_2)\Gamma(b_3)\Gamma(a_2-a_1)}{\Gamma(a_1)\Gamma(b_1-a_2)\Gamma(b_2-a_2)\Gamma(b_3-a_2)}(\tau/2)^{-2a_2}+\\
&\frac{\Gamma(b_1)\Gamma(b_2)\Gamma(b_3)(\tau)^{\frac{a_1+a_2-b_1-b_2-b_3+1/2}{2}}}{\pi^{1/2}\Gamma(a_1)\Gamma(a_2)}\cos(\tau).
\end{eqnarray}
\end{lem}

\subsection{Technical Results related to Space-Time Spectral Densities} \label{App_Spec}

\begin{thm}\label{Theo1}
Let $$ \Upsilon_{n}=\frac{(\eta)_n}{(\eta+\mu/2)_n (\eta+(\mu+1)/2)_n}$$ with $\eta=\frac{d+1}{2}+\kappa$. For  $\gamma>\max((d+3)/2,2\kappa+3)$,    and $\mu\geq (d+3)/2+\kappa+\alpha$, the isotropic spectral density, $\widehat{{\cal DGW}}_{\tt}$, associated with ${\cal DGW}_{\tt}$ as being defined at (\ref{WG22}), admits expression
 \begin{equation*}
 \begin{aligned}
\widehat{{\cal DGW}}_{\tt}(z,\tau; {\bX})=&\frac{\sigma^2 \beta^d L^{\boldsymbol{\varsigma}}}{2^{1/2}\pi^{3/2}}\sum_{n=0}^{\infty}\sum_{k=0}^{\infty}\frac{(-1)^n\Upsilon_n (z\beta/2)^n (\gamma(1+d+2n))_k\xi^{3/2}\tau^{-1/2}}{n! B^{-1}(5,1+k\delta)}\times\\
&\mathstrut_2 F_3\left(\frac{k\delta+1}{2},\frac{k\delta+2}{2};\frac{1}{2},\frac{6+k\delta}{2},\frac{7+k\delta}{2};-\left(\frac{\xi\tau}{2}\right)^2\right), \quad z,\tau \ge 0,
\end{aligned}
\end{equation*}
 where, with  $\boldsymbol{\varsigma}:= (\mu,\kappa,d)^{\top}$,
 \begin{equation*}
 c_{3}^{\boldsymbol{\varsigma}}=\frac{\Gamma(\mu+2\eta)}{\Gamma(\mu)},\qquad L^{\boldsymbol{\varsigma}}=\frac{K^{\boldsymbol{\varsigma}}\Gamma(\kappa)}{2^{1-\kappa}B(2\kappa,\mu+1)},\qquad
K^{\boldsymbol{\varsigma}}=\frac{2^{-\kappa-d+1}\pi^{-\frac{d}{2}}\Gamma(\mu+1)\Gamma(2\kappa+d)}{\Gamma(\kappa+\frac{d}{2})\Gamma(\mu+2\eta)}.
\end{equation*}
\end{thm}
\begin{proof}[Proof of Theorem~\ref{Theo1}]
A proof of the constructive type is provided. We denote $z:= \|\bz\|$ for $\bz \in \R^d$, and $r:= \|\br\|$ for $\br \in \R^d$, with $\|\cdot\|$ denoting the Euclidean norm. We start by writing the Fourier transform of the class $\widehat{{\cal DGW}}_{\tt}$ through the expression
\begin{equation}
\begin{aligned}
&\widehat{{\cal DGW}}_{\tt}(z,\tau; {\bX})\\
=&\sigma^2\int_{\R}\int_{\R^d}e^{- i (\langle u,\tau \rangle )-i(\langle\br,\bz \rangle)}{\cal C}(|u|; \xi,\delta,\gamma) \left ( 1- \frac{|u|}{\xi}\right )_+^{4}  {\cal GW}_{\beta, \mu,\kappa}\left ( \frac{\|\br\|}{{\cal C} (|u|;\xi,\beta,1)} \right ) \br{\rm d}u\\
=&\sigma^2\int_{\R} e^{- i \langle u,\tau \rangle } {\cal C}(|u|; \xi,\delta,\gamma) \left ( 1- \frac{|u|}{\xi}\right )_+^{4} \Bigg ( \int_{\R^d}e^{-i<\br,\bz>}{\cal GW}_{\beta,\mu,\kappa}\left ( \frac{\|\br\|}{{\cal C} (|u|;\xi,\beta,1)} \right ) {\rm d}\br \Bigg )  {\rm d} u. \\
\end{aligned}
\end{equation}
We now define the function ${\cal C}_{\tt}(\cdot;\xi,\delta,\gamma)$ through the identity
\begin{equation}
{\cal C}_{\tt}(\cdot;\xi,\delta,\gamma) = {\cal C}(\cdot;\xi,\delta,\gamma) \left (1 - \frac{t}{\xi}\right )_+^4  , \qquad t \ge 0.
\end{equation}
Then, we invoke Theorem 1 in \cite{Zastavnyi2006} in concert with basic Fourier calculus to rewrite the inner integral, and consequently, $\widehat{{\cal DGW}}_{\tt}$, as
\begin{equation} \label{shahid}
\begin{aligned}
&\widehat{{\cal DGW}}_{\tt}(z,\tau; {\bX})\\
=&\sigma^2 \beta^d L^\varsigma\int_{\R}e^{-i(\langle u,\tau \rangle )}  {\cal C}_{\tt}(|u|;\xi,\delta,\gamma+d) \mathstrut_1 F_2\left(\eta;\eta+\mu/2,\eta+(\mu+1)/2;-\big(\frac{z\beta {\cal C}(|u|; \xi,\delta,1) }{2} \big)^{2}\right){\rm d}u\\
=&\frac{\sigma^2 \beta^d L^\varsigma}{2\pi}\sum_{n=0}^{\infty}\frac{(-1)^n (\eta)_n (z\beta/2)^n}{n!(\eta+\mu/2)_n (\eta+(\mu+1)/2)_n}\int_{\R^+}t^{1/2}(1-t)_+^{4}(1+t^\delta)^{-2n-d-\gamma}\mathcal{J}_{-1/2}(t\tau) {\rm d} t. \\
\end{aligned}
\end{equation}

We now define  the Mellin-Barnes transformation \citep{borisov2006mellin}, given by the contour integral
\begin{equation*}\label{MellinB}
\frac{1}{(1+x)^{\alpha}} = \frac{1}{2\pi i}\frac{1}{\Gamma(\alpha)}
\oint_C ~x^u\Gamma(-u)\Gamma(\alpha+u) \text{d}u,
\end{equation*}
where
the contour $C$ contains the vertical line which passes between left and right poles in the complex plane $u$ from negative to positive imaginary infinity, and should be closed to the left in case $t > 1$, and to the right complex infinity if $0 < t < 1$.
In our case, the contour $C$ is given by $C=]-i\infty-\epsilon, i\infty-\epsilon[$, with $\epsilon>0$, and should be closed to the right complex infinity. We can now invoke (\ref{MellinB}) to rewrite the last line in (\ref{shahid}) as
\begin{equation}
\begin{aligned}
&\frac{\sigma^2 \beta^d L^{\boldsymbol{\varsigma}}}{2\pi}\sum_{n=0}^{\infty}\frac{(-1)^n \Upsilon_{n} (z\beta/2)^n}{n!}\int_{\R^+} t^{1/2}(1-t/\xi)_+^{4} \oint_C {\rm d}u\frac{\Gamma(-u)\Gamma\big(u+(\gamma+d+2n)\big)}{2\pi it^{u\delta}\xi^{u\delta}\Gamma(1+d+2n)}\mathcal{J}_{-1/2}(t\tau) {\rm d}t\\
=&\frac{\sigma^2 \beta^d L^{\boldsymbol{\varsigma}}}{2\pi}\sum_{n=0}^{\infty}\frac{(-1)^n \Upsilon_{n} (z\beta/2)^n}{n!}\oint_C {\rm d}u\frac{\Gamma(-u)\Gamma\big(u+(\gamma+d+2n)\big)}{2\pi i\xi^{u\delta}\Gamma(1+d+2n)}\int_{\R^+}t^{1/2+u\delta}(1-t/\xi)_+^{4} \mathcal{J}_{-1/2}(t\tau){\rm d}t.\\
\end{aligned}
\end{equation}

From  6.569 in \cite{grad} and, for $u\delta>-1$, we have
\begin{equation}
\begin{aligned}
&\int_{\R^+}t^{1/2+u\delta} \mathcal{J}_{-1/2}(t\tau)(1-t/\xi)_+^{4} {\rm d}t\\
&=\frac{(2/\tau)^{1/2}\Gamma(5)\Gamma(1+u\delta)}{\xi^{-(u\delta+3/2)}\Gamma(1/2)\Gamma(6+u\delta)}\mathstrut_2 F_3\left(\frac{u\delta+1}{2},\frac{u\delta+2}{2};\frac{1}{2},\frac{6+u\delta}{2},\frac{7+u\delta}{2};-\left(\frac{\xi\tau}{2}\right)^2 \right).
\end{aligned}
\end{equation}
We observe that there is only one function that contains poles in the
complex plane. This function is $\Gamma(-u)$ and gives  poles when  $u=k$,  with $k$ is a nonnegative integer. Then, the spectral density can be written as a double series and   we can write
\begin{eqnarray}
&&\widehat{\cal DGW}_{\tt}(z,\tau; {\bX}) \nonumber \\
&=&\frac{\sigma^2 \beta^d L^{\boldsymbol{\varsigma}}}{2\pi}\sum_{n=0}^{\infty}\frac{(-1)^n \Upsilon_{n} (z\beta/2)^n}{n!}\oint_C {\rm d}u\frac{\Gamma(-u)\Gamma\big(u+(\gamma+d+2n)\big)}{2\pi i\Gamma(\gamma+d+2n)}\times \nonumber \\
& &\frac{(2/\tau)^{1/2}\Gamma(5)\Gamma(1+u\delta)}{\xi^{-3/2}\Gamma(1/2)\Gamma(6+u\delta)}~\mathstrut_2 F_3\left(\frac{u\delta+1}{2},\frac{u\delta+2}{2};\frac{1}{2},\frac{6+u\delta}{2},\frac{7+u\delta}{2};-\left(\frac{\xi\tau}{2}\right)^2\right) \nonumber \\
&=&\frac{\sigma^2 \beta^d L^{\boldsymbol{\varsigma}}}{2\pi}\sum_{n=0}^{\infty}\frac{(-1)^n \Upsilon_{n} (z\beta/2)^n}{n!}\sum_{k=0}^{\infty}\frac{(-1)^k\Gamma(k+\gamma+d+2n)}{k!\Gamma(\gamma+d+2n)}\times \nonumber \\
& &\frac{(2/\tau)^{1/2}\Gamma(5)\Gamma(1+k\delta)}{\xi^{-3/2}\Gamma(1/2)\Gamma(6+k\delta)}~\mathstrut_2 F_3\left(\frac{k\delta+1}{2},\frac{k\delta+2}{2};\frac{1}{2},\frac{6+k\delta}{2},\frac{7+k\delta}{2};-\left(\frac{\xi\tau}{2}\right)^2\right) \nonumber \\
&=&\frac{\sigma^2 \beta^d L^{\boldsymbol{\varsigma}}}{2^{1/2}\pi^{3/2}}\sum_{n=0}^{\infty}\sum_{k=0}^{\infty}\frac{(-1)^{n+k}\Upsilon_n (z\beta/2)^n (\gamma+d+2n)_k\xi^{3/2}\tau^{-1/2}}{n!k! B^{-1}(5,1+k\delta)}\times \nonumber \\
& &\mathstrut_2 F_3\left(\frac{k\delta+1}{2},\frac{k\delta+2}{2};\frac{1}{2},\frac{6+k\delta}{2},\frac{7+k\delta}{2};-\left(\frac{\xi\tau}{2}\right)^2\right).
\end{eqnarray}
This completes the proof.
\end{proof}

Combining Lemmas~\ref{lem1}, \ref{lem2} and \ref{lem3}, we provide the following.
\begin{thm}\label{density1}
Let $\gamma>\max((d+3)/2,2\kappa+3)$  and $\mu\geq (d+3)/2+\kappa+\alpha$. Let $\widehat{\cal DGW}_{\tt}$ as being attained at Theorem \ref{Theo1}. Then, for $z,\tau \to \infty$, the following is true: \\
1. Let $$ D^k=\frac{\Gamma(1+k\delta/2)}{6+k\gamma)}, \qquad D_1^k=\frac{\pi\Gamma(1+(k\delta+4)/2)\Gamma((3+k\delta+4)/2)}{\Gamma(1+k\delta/2)\Gamma(-k\delta/2)\Gamma(5/2)\Gamma(3)},$$  
$$ D_2^k=\frac{\pi\Gamma(1+(k\delta+4)/2)\Gamma((7+k\delta)/2)}{\Gamma((k\delta+1)/2)\Gamma(-(k\delta+1)/2)\Gamma(2)\Gamma(3)}, \qquad F_1^k=\frac{2^{\mu-k-2\eta-1}\Gamma(\eta)\Gamma(d+\gamma+k)}{\Gamma(\mu)4(d+\gamma-2n)}, $$  \text{and} $$ D_3^k=\frac{\Gamma(1+(k\delta+4)/2)\Gamma((7+k\delta)/2)}{\Gamma((k\delta+1)/2)\Gamma(1+k\delta/2)}, \qquad F_2^{k}=\frac{2^{2\eta+\mu-d-k-\gamma-2}\Gamma(\frac{d+\gamma+k+1}{2})\Gamma(\eta-\frac{d+\gamma+k+1}{2})}{\Gamma(-\frac{k+1}{2})\Gamma(2\eta+\mu-d-\gamma-k-1)}.$$ Then,
\begin{equation*}
\begin{aligned}
\widehat{DGW}_{\tt}(z,\tau; {\bX})=&\frac{\sigma^2 2^{1/2-2\eta+\mu}\beta^d L^\varsigma\Gamma(2\eta+\mu)}{2\pi\Gamma(5)^{-1}\Gamma(\eta)}\sum_{k=0}^{\infty} \frac{(-1)^{k} D^k\Gamma((k+d+\gamma)/2)^{-1}}{k!\Gamma((1+k+d+\gamma)/2)} \\
&\Bigg [-D_1^k(\tau/2)^{-(1+k\delta)}+D_2^k(\tau/2)^{-(2+k\delta)}-D_3^k\tau^{-5}\cos(\tau)+ O(\tau^{-2})\Bigg ]\times\\
& \Bigg [\left(\frac{z}{2}\right)^{2\varrho}\Big(\cos(\pi\varrho+z)\left(1+O(z^{-2})\right)+\frac{\varrho_1}{z}\sin(\pi\varrho+z)\left(1+O(z^{-2})\right)\Big)+\\
&F_1^k \left(\frac{z}{2}\right)^{-2\eta}+F_2^k \left(\frac{z}{2}\right)^{-(d+k+1+\gamma)}\Bigg ]\left(1+O(z^{-2})\right),\\
\end{aligned}
\end{equation*}
2. For  $0<\delta< 1$, 
\begin{eqnarray*}
\widehat{\cal DGW}_{\tt}(z,\tau; {\bX})&\sim &\frac{\sigma^2 \beta^d L^\varsigma}{\sqrt{2\pi}}\Bigg [a^{\boldsymbol{\varsigma}}_2\left(\frac{\xi\tau}{2}\right)^{-\delta-1}(\beta~z)^{-2\eta}+O(\tau^{-2})O(z^{-2\eta})+O(\tau^{-\delta-1})O(z^{-2\eta-2})\\
&+&O(\tau^{-\delta-1})O(z^{-\eta-\mu})\Bigg],
\end{eqnarray*}
3. For  $\delta=1$,
\begin{eqnarray*}
\widehat{\cal DGW}_{\tt}(z,\tau; {\bX})&\sim &\frac{\sigma^2 \beta^d L^\varsigma}{\sqrt{2\pi}}\Bigg [\left(\frac{4a^{\boldsymbol{\varsigma}}_2}{\xi^{2}}+\frac{a^{\boldsymbol{\varsigma}}_1}{\xi}\right)\tau^{-2}(\beta~z)^{-2\eta}+O(\tau^{-2})O(z^{-2\eta-2})+O(\tau^{-2})O(z^{-\eta-\mu})\Bigg].
\end{eqnarray*}
4. Further, for $1<\delta<2$,
\begin{eqnarray*}
\widehat{\cal DGW}_{\tt}(z,\tau; {\bX})&\sim &\frac{\sigma^2 \beta^d L^\varsigma}{\sqrt{2\pi}}\Bigg [a_1\xi^{-1}\tau^{-2}(\beta~z)^{-2\eta}+O(\tau^{-\delta-1})O(z^{-2\eta})+O(\tau^{-2})O(z^{-2\eta-2})\\
&+&O(\tau^{-2})O(z^{-\eta-\mu})\Bigg],
\end{eqnarray*}
5. For $\delta=2$,
\begin{eqnarray*}
\widehat{\cal DGW}_{\tt}(z,\tau; {\bX})&\sim &\frac{\sigma^2 \beta^d L^\varsigma}{\sqrt{2\pi}}\Bigg [a^{\boldsymbol{\varsigma}}_1\xi^{-1}\tau^{-2}(\beta~z)^{-2\eta}+O(\tau^{-2})O(z^{-2\eta-2})+O(\tau^{-2})O(z^{-\eta-\mu})\Bigg],
\end{eqnarray*}
6.\label{point6} Finally, $$\widehat{\cal DGW}_{\tt}(z,\tau; {\bX})\asymp \tau^{-\delta-1}z^{-2\eta}\Unit_{(0<\delta\leq 1)}+\tau^{-2}z^{-2\eta}\Unit_{(1<\delta\leq 2)},$$
with $a^{\boldsymbol{\varsigma}}_1=5\sqrt{2/\pi}c_3^{\varsigma}$ \quad and \quad $a^{\boldsymbol{\varsigma}}_2=\frac{2^{\frac{1}{2}}c_3^{\varsigma}(\gamma+d-2\eta)\Gamma(\frac{\delta}{2}+1)}{\Gamma(-\frac{\delta}{2})\Gamma(6)^2}.$


\end{thm}

\begin{proof}
    We start by proving Point {\em 1.} From Theorem~\ref{Theo1}, direct inspection proves that 
\begin{equation}
\begin{aligned}
\widehat{\cal DGW}_{\tt}(z,\tau; {\bX})=&\frac{\sigma^2 \beta^d L^\varsigma}{2^{1/2}\pi^{3/2}}\sum_{n=0}^{\infty}\sum_{k=0}^{\infty}\frac{(-1)^{n+k}\Upsilon_n (z/2)^n (\gamma+d+2n)_k\xi^{3/2}\tau^{-1/2}}{n!k! B^{-1}(5,1+k\delta)}\times\\
&\mathstrut_2 F_3\left(\frac{k\delta+1}{2},\frac{k\delta+2}{2};\frac{1}{2},\frac{6+k\delta}{2},\frac{7+k\delta}{2};-\left(\frac{\xi\tau}{2}\right)^2\right). \\
\end{aligned}
\end{equation}
Next, for $\tau$ large enough, we have
\begin{equation}
\begin{aligned}
\widehat{\cal DGW}_{\tt}(z,\tau; {\bX})=&\frac{\sigma^2 \beta^d L^\varsigma}{2^{1/2}\pi^{3/2}}\sum_{n=0}^{\infty}\sum_{k=0}^{\infty}\frac{(-1)^{n+k}\Upsilon_n (z/2)^n (\gamma+d+2n)_k\xi^{3/2}\tau^{-1/2}}{n!k! B^{-1}(5,1+k\delta)}\times\Bigg [\\
&-\frac{\pi\Gamma(1+(k\delta+4)/2)\Gamma((7+k\delta)/2)}{\Gamma(1+k\delta/2)\Gamma(-k\delta/2)\Gamma(5/2)\Gamma(3)}(\tau/2)^{-(1+k\delta)}+\\
&\frac{\pi\Gamma(1+(k\delta+4)/2)\Gamma((7+k\delta)/2)}{\Gamma((k\delta+1)/2)\Gamma(-(k\delta+1)/2)\Gamma(4)\Gamma(2)}(\tau/2)^{-(2+k\delta)}-\\
&\frac{\Gamma(1+(k\delta+4)/2)\Gamma((7+k\delta)/2)}{\Gamma((k\delta+1)/2)\Gamma(1+k\delta/2)}\tau^{-5}\cos(\tau)+ O(\tau^{-2})\Bigg ].
\end{aligned}
\end{equation}
Hence, {for $z$ large enough we have}
\begin{equation*}
\begin{aligned}
\widehat{\cal DGW}_{\tt}(z,\tau; {\bX})=&\frac{\sigma^2 \beta^d L^\varsigma \xi^{3/2}}{2^{1/2}\pi^{3/2}}\sum_{k=0}^{\infty}\frac{(-1)^{k} \tau^{-1/2}}{k! B^{-1}(5,1+k\delta)}\times \Bigg (\sum_{n=0}^{\infty}\frac{(-1)^n\Upsilon_n (z/2)^n (\gamma+d+2n)_k}{n!} \Bigg ) \times \\
&\Bigg [ -D_1^k(\tau/2)^{-(1+k\delta)}+D_2^k(\tau/2)^{-(2+k\delta)} 
-D_3^k\tau^{-5}\cos(\tau)+ O(\tau^{-2})\Bigg ] .
\end{aligned}
\end{equation*}
Applying Lemma~\ref{lem3}, we have
\begin{equation*}
\begin{aligned}
\widehat{\cal DGW}_{\tt}(z,\tau; {\bX})=&2^{2k}(d+\gamma)_k\frac{\sigma^2 \beta^d L^\varsigma \xi^{3/2}}{2^{1/2}\pi^{3/2}}\sum_{k=0}^{\infty}\frac{(-1)^{k} \tau^{-1/2}}{k! B^{-1}(5,1+k\delta)}\times\\
&\Bigg [-D_1^k(\tau/2)^{-(1+k\delta)}+D_2^k(\tau/2)^{-(2+k\delta)}-D_3^k\tau^{-5}\cos(\tau)+ O(\tau^{-2})\Bigg ] \times\\
& \mathstrut_3 F_4\left(\eta,\frac{d+\gamma+k}{2},\frac{d+\gamma+k+1}{2};
\frac{d+\gamma}{2},\frac{d+\gamma+1}{2},\eta+\frac{\mu}{2},\eta+\frac{\mu+1}{2};-\left(\frac{z}{2}\right)^2\right).\\
\end{aligned}
\end{equation*}
The proof {of the first Point} is completed by invoking Lemma~\ref{lem2} and through direct inspection.\\
The proof of the {other Points} is given  as follow. For a higher frequency of $z$, we have
\begin{equation*} \label{shahid1}
\begin{aligned}
&\widehat{\cal DGW}_{\tt}(z,\tau; {\bX})\\
=&\frac{\sigma^2 \beta^d L^\varsigma}{\xi^{-3/2}}\int_{\R^+} t^{1/2} {\cal C}(t;1,\delta, \gamma+d) J_{-\frac{1}{2}}(t\xi\tau)\left[ c_3^{\varsigma}\left(\frac{z\beta}{(1+t^{\delta})}\right)^{-2\eta}\left\{1+(1+t^{\delta})^2 O(z^{-2})\right\}+(1+t^{\delta})^{\mu+\eta}O(z^{-(\mu+\eta)})\right]{\rm d}t\\
\end{aligned}
\end{equation*}
By a simple change of variable $u=t\tau$, we obtain
\begin{equation*} \label{shahid2}
\begin{aligned}
&\widehat{\cal DGW}_{\tt}(z,\tau; {\bX})\\
=&\frac{\sigma^2 \beta^d L^\varsigma}{(\xi^{-1}\tau)^{3/2}}\int_{0}^{\tau} u^{1/2} {\cal C}(u/\tau; 1 \delta,\gamma+d) J_{-\frac{1}{2}}(u\xi)\Bigg[ c_3^{\varsigma}\left(\frac{z\beta}{(1+(u/\tau)^{\delta})}\right)^{-2\eta}\left\{1+(1+(u/\tau)^{\delta})^2 O(z^{-2})\right\}\\
&+(1+(u/\tau)^{\delta})^{\mu+\eta}O(z^{-(\mu+\eta)})\Bigg]{\rm d}u\\
=&\frac{\sigma^2 \beta^d L^\varsigma}{(\xi^{-1}\tau)^{3/2}}\Bigg[ c_3^{\varsigma}(z\beta)^{-2\eta}\left\{\int_{0}^{\tau} u^{1/2} {\cal C}(u/\tau; 1,\delta, \gamma+2 - 2 \eta)J_{-\frac{1}{2}}(u\xi){\rm d}u+O(z^{-2})\int_{0}^{\tau} u^{1/2} {\cal C}(u/\tau; 1,\delta, \gamma+2 - 2 \eta-2){\rm d}u\right\}\\
&+O(z^{-(\mu+\eta)})\int_{0}^{\tau} u^{1/2} {\cal C}(u/\tau; 1,\delta, \gamma+d-\mu-\eta)J_{-\frac{1}{2}}(u\xi){\rm d}u\Bigg]. \\
\end{aligned}
\end{equation*}
Next, for $\tau$ large enough, we write

\begin{equation}\label{eq121}
 \left(1+\frac{u^{\delta}}{\tau^\delta}\right)^{-\gamma}= 1-\frac{\Gamma(\gamma+1)}{\Gamma(\gamma)}\frac{u^\delta}{\tau^\delta}+O \Big (\frac{u^{2\delta}}{\tau^{2\delta}} \Big).
\end{equation}
{Then, using Equation~\ref{eq121}, we obtain the following expression}
\begin{eqnarray}\label{EqAsymp}
\widehat{\cal DGW}_{\tt}(z,\tau; {\bX})&\sim &\frac{\sigma^2 \beta^d L^\varsigma}{\sqrt{2\pi}}\Bigg[\left\{a^{\boldsymbol{\varsigma}}_1\xi^{-1}\tau^{-2}+O(\tau^{-4})+O(\tau^{-5})\right\}(\beta~z)^{-2\eta} \nonumber\\
&+&\big\{a^{\boldsymbol{\varsigma}}_2\left(\frac{\tau\xi}{2}\right)^{-(\delta+1)}+ O(\tau^{-(\delta+2)})+O(\tau^{-(2\delta+1)})\big\}(\beta~z)^{-2\eta} \nonumber\\
&+& \left(O(\tau^{-2})+O(\tau^{-(\delta+1)})\right)\left(O(z^{-2(\eta+1)})+O(z^{-(\mu+\eta)})\right)\Bigg].
\end{eqnarray}
{ From Equation~(\ref{EqAsymp}), it is easy to verify that the Point~2-5 of Theorem \ref{density1} are satisfied.}
%
%
\end{proof}




\section{Proofs of Compatibility Theorems} \label{App_Comp}
For the reminder of the proofs, we denote with $P({{\cal DM}}(\btheta))$,
$P({\cal DGW}({\bX}))$,
 zero mean Gaussian measure induced by ${{\cal DM}}(\cdot,\cdot;\btheta)$ and ${\cal DGW}(\cdot,\cdot;{\bX})$ covariance functions, respectively. Analogous notation is used for $P({\cal DGW}_{\tt}({\bX}))$.

\begin{proof}[Proof of Theorem \ref{TW_vs_TW}]

We need to find conditions such that  for some positive and finite $c$,
\begin{equation}\label{eq:98}
\int_{{\cal A}}z^{d-1} \bigg(
\frac{\widehat{\cal DGW}_{\tt}(z,\tau; {\bX}_1)-\widehat{\cal DGW}_{\tt}(z,\tau;{\bX}_0)}{\widehat{\cal DGW}_{\tt}(z,\tau;{\bX}_0)}
 \bigg )^{2} {\rm d} z {\rm d}\tau<\infty,
\end{equation}
 for  ${\cal A}$ depending on $c$  as specified through  \eqref{spectralfinite2_bis}. We easily   show that $\widehat{K}_{\tt}(z,\tau;{\bX}_0) z^{2\eta}\tau^{2}$ is bounded away from $0$ and $\infty$ as $z,\tau \to \infty$, with $z/\tau$ converging to a constant $k$.
Using Theorem~\ref{AB2} (Points~{\it 3} and ~{\it 5}) when $\delta\in \{1,2\}$, we have, as $z,\tau\to\infty$,

\begin{equation*}
\begin{aligned}
&\Bigg|\frac{\widehat{\cal DGW}_{\tt}(z,\tau; {\bX}_1)-\widehat{\cal DGW}_{\tt}(z,\tau;{\bX}_0)}{\widehat{\cal DGW}_{\tt}(z,\tau;{\bX}_0)}\Bigg|\\
&=\left(\tau^{\delta+1}z^{+2\eta}\Unit_{(0<\delta\leq 1)}+\tau^{2}z^{2\eta}\Unit_{(1<\delta\leq 2)}\right)\Bigg|\frac{\sigma^2 \beta^d L^\varsigma}{\sqrt{2\pi}}\Bigg [\left(\frac{4a^{\boldsymbol{\varsigma}}_2}{\xi^{2}}+
\frac{a^{\boldsymbol{\varsigma}}_1}{\xi}\right)\tau^{-2}(\beta~z)^{-2\eta}+O(\tau^{-2})O(z^{-2\eta-2})\\
&+O(\tau^{-2})O(z^{-\eta-\mu})\Bigg]+\frac{\sigma^2 \beta^d L^\varsigma}{\sqrt{2\pi}}\Bigg [\frac{a^{\boldsymbol{\varsigma}}_1}{\xi}\tau^{-2}(\beta~z)^{-2\eta}+O(\tau^{-2})O(z^{-2\eta-2})+O(\tau^{-2})O(z^{-\eta-\mu})\Bigg]\mathbf{1}_{\delta=2} \\ 
&- \sigma_0^2\beta_0^d
L^{\boldsymbol{\varsigma},0}c_{3}^{\boldsymbol{\varsigma},0}(z\beta_0)^{-2\eta}\bigg(\left[\varrho_{\gamma_0,\eta}\tau^{-(1+\delta_0)}-\mathcal{O}\left(\tau^{-(1+2\delta_0)}\right)\right]\\
&+\left[\varrho_{\gamma_0,\eta+1}\tau^{-(1+\delta_0)}-\mathcal{O}\left(\tau^{-(1+2\delta_0)}\right)\right]\mathcal{O}(z^{-2})\bigg)+\left[\varrho_{\gamma_0,0}\tau^{-(1+\delta_0)}-\mathcal{O}\left(\tau^{-(1+2\delta_0)}\right)\right]\mathcal{O}(z^{-(\mu+\eta)}) \Bigg|.
\end{aligned}
\end{equation*}


Then, 
 \begin{enumerate}
\item  for $0<\delta_0\leq 1$, $2(\kappa_1-\kappa_0)=\delta_0-1$,  ${\cal DGW}_{\tt}(\cdot,\cdot; {\bX}_1)$ and ${\cal DGW}_{\tt}(z,\tau; {\bX}_0)$ are compatible if and only if 
 $$
\left(\frac{4a^{\boldsymbol{\varsigma}}_{2,1}}{\xi_1^{2}}+\frac{a^{\boldsymbol{\varsigma}}_{1,1}}{\xi_1}\right)\frac{\sigma_1^2L^{\boldsymbol{\varsigma},1}}{\beta_1^{2 \kappa_1+1}}= \left(\frac{4a^{\boldsymbol{\varsigma}}_{2,0}}{\xi_0^{2}}+\frac{a^{\boldsymbol{\varsigma}}_{1,0}}{\xi_0}\right)\frac{\sigma_0^2L^{\boldsymbol{\varsigma},0}}{\beta_0^{2 \kappa_0+1}}.
$$
\item for  $1<\delta_0\leq 2$, $\kappa_1=\kappa_0$, ${\cal DGW}_{\tt}(\cdot,\cdot; {\bX}_1)$ and ${\cal DGW}_{\tt}(z,\tau; {\bX}_0)$ are compatible if and only if $$
\left(\frac{4a^{\boldsymbol{\varsigma}}_{2,1}}{\xi_1^{2}}+\frac{a^{\boldsymbol{\varsigma}}_{1,1}}{\xi_1}\right)\frac{\sigma_1^2}{\beta_1^{2 \kappa_1+1}}= \left(\frac{4a^{\boldsymbol{\varsigma}}_{2,0}}{\xi_0^{2}}+\frac{a^{\boldsymbol{\varsigma}}_{1,0}}{\xi_0}\right)\frac{\sigma_0^2}{\beta_0^{2 \kappa_0+1}}.
$$
\end{enumerate}

\end{proof}

\begin{proof}[Proof of Theorem \ref{TAPvsDW}]
We need to find conditions such that  for some positive and finite $c$,
\begin{equation}\label{eq:99}
\int_{{\cal A}}z^{d-1} \bigg(
\frac{\widehat{\cal DGW}_{\tt}(z,\tau; {\bX}_1)-\widehat{\cal DGW}(z,\tau;{\bX}_0)}{\widehat{\cal DGW}(z,\tau;{\bX}_0)}
 \bigg )^{2} {\rm d} z {\rm d}\tau<\infty,
\end{equation}
 for  ${\cal A}$ depending on $c$  as specified through  \eqref{spectralfinite2_bis}. Arguments in \cite{Ryan17} show that ${\cal DGW}(z,\tau;{\bX}_0) z^{2\eta}\tau^{\delta+1}$ is bounded away from $0$ and $\infty$ as $z,\tau \to \infty$, with $z/\tau$ converging to a constant $k$.
Using Theorem~\ref{AB2} (Points~{\it 4} and ~{\it 5}) when $1<\delta\leq 2$, we have, as $z,\tau\to\infty$,
\begin{equation*}
\begin{aligned}
&\Bigg|\frac{\widehat{\cal DGW}_{\tt}(z,\tau; {\bX}_1)-\widehat{\cal DGW}(z,\tau;{\bX}_{0})}{\widehat{\cal DGW}(z,\tau;{\bX}_{0})}\Bigg|\\
&=\frac{\beta_0^{2\eta_0-d}z^{2\eta_0}\tau^{1+\delta_0}}{\sigma_0^2 \varrho_{\gamma_0,\eta} c_3^{\varsigma,0}L^{\varsigma,0}}\Bigg|\frac{\sigma^2 \beta^d L^{\varsigma,1}}{\sqrt{2\pi}}\Bigg [a^{\boldsymbol{\varsigma}}_1\xi^{-1}\tau^{-2}(\beta~z)^{-2\eta}+O(\tau^{-2})O(z^{-2\eta-2})+O(\tau^{-\delta_1-1})O(z^{-2\eta})+O(\tau^{-2})O(z^{-\eta-\mu})\Bigg]\times\\
&\mathbf{1}_{1<\delta_1<2}+\frac{\sigma^2 \beta^d L^{\varsigma,1}}{\sqrt{2\pi}}\Bigg [a^{\boldsymbol{\varsigma}}_1\xi^{-1}\tau^{-2}(\beta~z)^{-2\eta}+O(\tau^{-2})O(z^{-2\eta-2})+O(\tau^{-2})O(z^{-\eta-\mu})\Bigg]\mathbf{1}_{\delta_1=2} \\ 
&- \sigma_0^2\beta_0^d
L^{\boldsymbol{\varsigma},0}c_{3}^{\boldsymbol{\varsigma},0}(z\beta_0)^{-2\eta_0}\bigg(\left[\varrho_{\gamma_0,\eta_0}\tau^{-(1+\delta_0)}-\mathcal{O}\left(\tau^{-(1+2\delta_0)}\right)\right]\\
&+\left[\varrho_{\gamma_0,\eta_0+1}\tau^{-(1+\delta_0)}-\mathcal{O}\left(\tau^{-(1+2\delta_0)}\right)\right]\mathcal{O}(z^{-2})\bigg)+\left[\varrho_{\gamma_0,0}\tau^{-(1+\delta_0)}-\mathcal{O}\left(\tau^{-(1+2\delta_0)}\right)\right]\mathcal{O}(z^{-(\mu+\eta)}) \Bigg|,
\end{aligned}
\end{equation*}

Then, for $\delta_1=2$, if {$\delta_0=1+2(\kappa_1-\kappa_0)$ } and
$$ \frac{\sigma_0^2\varrho_{\gamma,\eta_0}c_3^{\varsigma,0}L^{\varsigma,0} }{\beta_0^{-(1+2\kappa_0)}}=\frac{\sigma_1^2 a^{\boldsymbol{\varsigma}}_1 \xi^{-1}L^{\varsigma,1}}{\sqrt{2\pi} \beta_1^{-(1+2\kappa_1)}}, $$ 
 the compatibility of the two models for any
bounded infinite set $D\times {\cal T}\subset \R^d\times\R$, $d=1,2$, holds.

\end{proof}

\begin{proof}[Proof of Theorem~\ref{TGWDvsMat}]
We need to find conditions such that  for some positive and finite $c$,
\begin{equation}\label{eq:999}
\int_{{\cal A}}z^{d-1} \bigg(
\frac{\widehat{\cal DGW}_{\tt}(z,\tau;{\bX}_0)-\widehat{\cal DM}(z,\tau;{\btheta}_0)}{\widehat{\cal DM}(z,\tau;{\btheta}_0)}
 \bigg )^{2} {\rm d} z {\rm d}\tau<\infty,
\end{equation}
 for  ${\cal A}$ depends on $c$  as specified through  \eqref{spectralfinite2_bis}. It is known that $\widehat{\cal DM}(z,\tau;{\btheta}_0) z^{2\nu}$ is bounded away from 0 and $\infty$ as $z,\tau \to \infty$, with $z/\tau$ converging to a constant $k$ \citep{Ryan17}.
Using Theorem~\ref{density1} (Points~{\it 3, 5}) and  Theorem~2 of \citep{faouzi2022space} (Point~{\it2}) (Point~{\it 2}) when $\epsilon\in(0,1]$, we have, as $z,\tau\to\infty$,
\begin{equation*}
\begin{aligned}
\Bigg|&\frac{\widehat{\cal DGW}_{\tt}(z,\tau; {\bX}_0)-\widehat{\cal DM}(z,\tau;{\btheta}_0)}{\widehat{\cal DM}(z,\tau;{\btheta}_0)}\Bigg|
=\Bigg|\frac{\widehat{\cal DGW}_{\tt}(z,\tau; {\bX}_0)}{\widehat{\cal DM}(z,\tau;{\btheta}_0)}-1\Bigg| \\
&=\Bigg|\ell(\btheta_0)^{-1}(\epsilon  z\tau)^{2\nu}\Bigg \{\frac{\sigma^2 \beta^d L^\varsigma}{\sqrt{2\pi}}\Big [\left(\frac{4a^{\boldsymbol{\varsigma}}_2}{\xi^{2}}+\frac{a^{\boldsymbol{\varsigma}}_1}{\xi}\right)\tau^{-2}(\beta~z)^{-2\eta}+O(\tau^{-2})O(z^{-2\eta-2})+O(\tau^{-2})O(z^{-\eta-\mu})\Big]\times\\
&\mathbf{1}_{\delta=1}+\frac{\sigma^2 \beta^d L^\varsigma}{\sqrt{2\pi}}\Big [a^{\boldsymbol{\varsigma}}_1\xi^{-1}\tau^{-2}(\beta~z)^{-2\eta}+O(\tau^{-2})O(z^{-2\eta-2})+O(\tau^{-2})O(z^{-\eta-\mu})\Big]\mathbf{1}_{\delta=2}\Bigg \} \\ &\times \left(1+\frac{\nu\zeta^2\upsilon^2}{\epsilon^2 z^2\tau^2}+\frac{\nu\upsilon^2}{\epsilon^2 \tau^2}+
    \frac{\nu\zeta^2}{\epsilon^2 z^2}+\mathcal{O}(\tau^{-4}z^{-4})\right)-1\Bigg |,
\end{aligned}
\end{equation*}
 for $\epsilon\in (0,1]$, $\kappa>\frac{d-1}{2}$ and, $d=1,2$ if  $2\nu=\eta+1$,  $\mu>\max{(\eta+1+\alpha,\eta+\frac{1}{2})}$
and $$ \Big[\frac{\sigma^2  L^{\varsigma}}{\sqrt{2\pi}\beta^{1+2\kappa}}\left(\frac{4a^{\boldsymbol{\varsigma}}_2}{\xi^{2}}+\frac{a^{\boldsymbol{\varsigma}}_1}{\xi}\right)\Big]\mathbf{1}_{\delta=1}=\ell(\btheta_0)\epsilon^{-2\nu},$$ so that~\eqref{eq:999} holds. \hfill

\end{proof}

\section{Proving consistency of ML estimates under compatibility} \label{App_Consistency}

Let $D\times{\cal T}$ be a bounded subset of $ \R^d\times\R$ and let~ \\ $\bZ_{nm}=(Z(\boldsymbol{s}_1,t_1),\ldots,Z(\boldsymbol{s}_n,t_m))^{\top}$
be a finite  realization of a zero mean stationary Gaussian random field $Z(\boldsymbol{s},t)$, $(\boldsymbol{s},t)\in D \times {\cal T}$,  with  a given parametric covariance function
$\sigma^2 K(r,t; \btau)$, with $\sigma^2>0$, $\btau$ a parameter vector and $K(0,0; \btau)=1$ for all $\btau$. Here we consider the ${\T}$ covariance  model, that is
 $K_{\T}(r,t;{\bX}_0)= \sigma^2 K(r,t; \btau)$, where
 $K(r,t; \btau)={\cal DGW}_{\tt}(\cdot,\cdot; \bX)$
 with  $\btau=(\mu,\kappa,\beta,\delta,\xi,\gamma)^{\top}$. At the same time, in the current exposition $\btau$ will not contain the parameters that are fixed, but only those that are to be estimated through ML.
Specifically,  $\delta$, $\gamma$ and $\mu$ are assumed known and  fixed, that is we assume
$\btau=(\kappa,\beta,\xi)^{\top}$, the spatial and temporal scale parameters.


Then, the   Gaussian log-likelihood function is defined as:
\begin{equation}\label{eq:17}
\mathcal{L}_{nm}(\sigma^{2},\btau)=-\frac{1}{2} \left(nm\log(2\pi\sigma^{2})+\log\left(|R_{nm}(\btau)|\right)+\frac{1}{\sigma^{2}}\bZ_{nm}^{\top}R_{nm}(\btau)^{-1}\bZ_{nm} \right),
\end{equation}
where  
$R_{nm}(\btau)=[K(\|\boldsymbol{s}_i-\boldsymbol{s}_j\|,|t_l-t_k|; \btau)]_{i,j=1;,l,k=1}^{n;m}$ is the correlation matrix.


Let $\hat{\sigma}^2_{nm}$ be the ML estimator  of the variance parameter obtained by maximizing
$\mathcal{L}_{nm}(\sigma^{2},\btau)$ with respect to $\sigma^2$, and given by
\begin{equation}\label{sigma1}
\hat{\sigma}^2_{nm}(\btau)=\frac{1}{nm}\bZ_{nm}^{\top}R_{nm}(\btau)^{-1}\bZ_{nm}.
\end{equation}


We now
establish strong consistency and asymptotic distribution of the random variable $\hat{\sigma}^2_{nm}(\btau)/\beta^{2 \kappa+1}$
{\em i.e.}, the ML estimator of the microergodic parameter.

\begin{thm}\label{theo9}
Let $Z(\boldsymbol{s},t)$, $(\boldsymbol{s},t)\in D\times{\cal T}\subset \R^d\times\R$, $d=1,2$, be a zero mean Gaussian random field  with  covariance model
 $K_{\T}(r,t;{\bX_0})= \sigma_0^2 K(r,t; \btau)$, with $\btau=(\mu,\kappa,\beta,\delta,\xi,\gamma)^{\top}$,  $\kappa \ge 0$,  $\mu>\max{(\eta+\frac{d+1}{2},\eta+1+\alpha)}$ and $\delta=1$. For   $\gamma$ and $\mu$ fixed and known and arbitrary $\kappa$, $\beta$, $\xi$ , we have, as $n,m\to \infty$,
\begin{enumerate}
\item $\left(\frac{4a^{\boldsymbol{\varsigma}}_{2}}{\xi^{2}}+\frac{a^{\boldsymbol{\varsigma}}_{1}}{\xi}\right)\frac{\sigma_{nm}^2(\btau)L^{\boldsymbol{\varsigma}}}{\beta^{2 \kappa+1}}\stackrel{a.s.}{\longrightarrow} \left(\frac{4a^{\boldsymbol{\varsigma}}_{2,0}}{\xi_0^{2}}+a^{\boldsymbol{\varsigma}}_{1,0}\right)\frac{\sigma_0^2L^{\boldsymbol{\varsigma},0}}{\beta_0^{2 \kappa_0+1}}$, and
\item $\sqrt{n\times m}\left(\left(\frac{4a^{\boldsymbol{\varsigma}}_{2}}{\xi^{2}}+\frac{a^{\boldsymbol{\varsigma}}_{1}}{\xi}\right)\frac{\sigma_{nm}^2(\btau)L^{\boldsymbol{\varsigma}}}{\beta^{2 \kappa+1}}- \left(\frac{4a^{\boldsymbol{\varsigma}}_{2,0}}{\xi_0^{2}}+a^{\boldsymbol{\varsigma}}_{1,0}\right)\frac{\sigma_0^2L^{\boldsymbol{\varsigma},0}}{\beta_0^{2 \kappa_0+1}}\right)\stackrel{\mathcal{D}}{\longrightarrow} \mathcal{N}\left(0,2\left( \left(\frac{4a^{\boldsymbol{\varsigma}}_{2,0}}{\xi_0^{2}}+a^{\boldsymbol{\varsigma}}_{1,0}\right)\frac{\sigma_0^2L^{\boldsymbol{\varsigma},0}}{\beta_0^{2 \kappa_0+1}}\right)^2\right)$.
\end{enumerate}
\end{thm}

\bibliographystyle{apalike}
\bibliography{mybibPT.bib}

\end{document}